\documentclass{gtmon_a}
\pdfoutput=1

\usepackage[all]{xy}

\proceedingstitle{Proceedings of the Nishida Fest (Kinosaki 2003)}
\conferencestart{28 July 2003}
\conferenceend{8 August 2003}
\conferencename{International Conference in Homotopy Theory}
\conferencelocation{Kinosaki, Japan}

\editor{Matthew Ando}
\givenname{Matthew}
\surname{Ando}

\editor{Norihiko Minami}
\givenname{Norihiko}
\surname{Minami}

\editor{Jack Morava}
\givenname{Jack}
\surname{Morava}

\editor{W Stephen Wilson}
\givenname{W Stephen}
\surname{Wilson}

\title{Milnor operations and the generalized Chern character}

\author{Takeshi Torii}
\givenname{Takeshi}
\surname{Torii}
\address{Department of Applied Mathematics\\
Fukuoka University\\\newline
Fukuoka 814-0180\\
Japan}
\email{torii@math.sci.fukuoka-u.ac.jp}
\urladdr{}

\dedicatory{Dedicated to Professor Nishida on the occasion of his 60th birthday}

\volumenumber{10}
\issuenumber{}
\publicationyear{2007}
\papernumber{21}
\startpage{383}
\endpage{421}

\doi{}
\MR{}
\Zbl{}

\arxivreference{}


\received{24 September 2004}
\revised{30 May 2005}
\accepted{}
\published{18 April 2007}
\publishedonline{18 April 2007}
\proposed{}
\seconded{}
\corresponding{}
\version{}

\makeatletter
\def\cnewtheorem#1[#2]#3{\newtheorem{#1}{#3}[section]
\expandafter\let\csname c@#1\endcsname\c@theorem}

\let\xysavmatrix\xymatrix
\def\xymatrix{\disablesubscriptcorrection\xysavmatrix}
\AtBeginDocument{\let\bar\wbar\let\tilde\wtilde\let\hat\what}
\def\tmf{tm\!f\!}

\newtheorem{theoremA}{Theorem}

\makeautorefname{theoremA}{Theorem}
\newtheorem{theorem}{Theorem}[section]
\cnewtheorem{proposition}[theorem]{Proposition}
\cnewtheorem{lemma}[theorem]{Lemma}
\cnewtheorem{corollary}[theorem]{Corollary}
\cnewtheorem{remark}[theorem]{Remark}

\theoremstyle{definition}
\cnewtheorem{definition}[theorem]{Definition}
\cnewtheorem{notation}[theorem]{Notation}
\cnewtheorem{example}[theorem]{Example}
\cnewtheorem{question}[theorem]{Question}

\makeatother  

\def\subrel#1#2{\mathrel{\mathop{#2}\limits_{#1}}}

\def\fp#1{{\mathbf F}_{\! p^{#1}}}

\def\zp#1{{\mathbf Z}/{p^{#1}}}
\def\z2{{\mathbf Z}/2}

\def\cp{{\mathbf C}P^{\infty}}

\def\hz2{H\z2}

\def\pn{P(n)}

\def\power#1{[\![#1]\!]}

\def\hom#1{\mbox{\rm Hom}_{#1}}

\def\gal#1{\mbox{\rm Gal}(#1)}

\def\inverselimit#1{\subrel{\subrel{#1}{\longleftarrow}}{\lim}}
\def\directlimit#1{\subrel{\subrel{#1}{\longrightarrow}}{\lim}}

\def\pt{\mbox{\rm pt}}

\def\pn{P(n)}

\def\ee{E^{\vee}_*(E)}
\def\kk{K_*(K)}

\def\fmod#1{\mbox{\rm FMod}_{#1}}
\def\fmodc#1{\mbox{\rm FMod}_{#1}^c}
\def\profg#1{\mbox{\rm ProFG}_{#1}}
\def\falgc#1{{\mbox{\rm FAlg}^c_{#1}}}

\def\wtens{\widehat{\otimes}}

\def\acha{A^{\vee}_*(A)}
\def\F{\mathbf F}
\def\Gal{\mbox{\rm Gal}}
\def\M{M^{\mbox{\scriptsize $\gal{\F/\fp{n}}$}}}

\begin{document}

\begin{htmlabstract}
We have shown that the n-th Morava K&ndash;theory K<sup>*</sup>(X) for a
CW&ndash;spectrum X with action of Morava stabilizer group G<sub>n</sub>
can be recovered from the system of some height (n+1) cohomology groups
E<sup>*</sup>(Z) with G<sub>n+1</sub>&ndash;action indexed by finite
subspectra Z.  In this note we reformulate and extend the above result.
We construct a symmetric monoidal functor F from the category of
E<sup>&or;</sup><sub>*</sub>(E)&ndash;precomodules to the category of
K<sub>*</sub>(K)&ndash;comodules.  Then we show that K<sup>*</sup>(X)
is naturally isomorphic to the inverse limit of F(E<sup>*</sup>(Z))
as a K<sub>*</sub>(K)&ndash;comodule.
\end{htmlabstract}

\begin{abstract}
We have shown that the $n$--th Morava $K$--theory $K^*(X)$ for a
CW--spectrum $X$ with action of Morava stabilizer group $G_n$ can be
recovered from the system of some height--$(n+1)$ cohomology groups
$E^*(Z)$ with $G_{n+1}$--action indexed by finite subspectra $Z$.
In this note we reformulate and extend the above result.  We construct
a symmetric monoidal functor ${\mathcal F}$ from the category of
$E^{\vee}_*(E)$--precomodules to the category of $K_{*}(K)$--comodules.
Then we show that $K^*(X)$ is naturally isomorphic to the inverse limit
of ${\mathcal F}(E^*(Z))$ as a $K_{*}(K)$--comodule.
\end{abstract}

\maketitle

\section{Introduction}
\label{section:introduction}

From the chromatic point of view,
the complex $K$--theory is a height--$1$ cohomology theory
and the ordinary rational cohomology is a height--$0$ cohomology theory.
As geometric aspects,
the rational cohomology is defined by means of differential forms,
and the $K$--theory is defined by means of vector bundles.
The classical Chern character associates 
to a complex vector bundle 
the sum of exponentials of 
formal roots of the total Chern polynomial.
It may be regarded as a multiplicative
natural transformation from the $K$--theory
to the rational cohomology,
that is to say,
from a height--$1$ cohomology to a height--$0$ cohomology.
There is a height--$2$ cohomology theory, which is called
the elliptic cohomology.
Conjecturally,
the elliptic cohomology may also have a geometric interpretation
analogous to the rational cohomology and the $K$--theory.
A generalization of Chern character to 
the elliptic cohomology has been considered by Miller~\cite{Miller}. 
The idea is that 
the formal group law on the moduli stack of elliptic curves
is degenerate to the multiplicative formal group law
when it is restricted around a cusp. 
Miller's elliptic character is a multiplicative
natural transformation from the elliptic cohomology
to the $K$--theory with coefficients in the formal
Laurent series ring,
hence from a height--$2$ theory to a height--$1$ theory.

The elliptic character may be regarded as 
the $q$--expansion map of modular forms
parametrized by spaces. 
The $q$--expansion is the Fourier
expansion of modular forms at a cusp,
which associates a formal Laurent series
with variable $q=\exp (2\pi \sqrt{-1}\tau)$.
The $q$--expansion map has been extended at more general 
algebraic setting,
and it has been shown that it has a very good property,
which is called $q$--expansion principle
(cf Deligne--Rapoport~\cite{D-R} and Katz~\cite{Katz}).
In particular, the $q$--expansion map is injective,
and hence the modular forms are controlled
by their $q$--expansion.
The analogous property on the elliptic character has been studied 
at odd primes by Laures~\cite{Laures-1}.
At the prime $2$,
there is a more elaborate cohomology theory
related to elliptic curves and modular forms.
It is defined by the spectrum $tm\!f\!$ of topological modular forms,
which is introduced by Mike Hopkins.  
The $q$--expansion map (evaluation at the Tate curve)
also induces a ring spectrum map from $\tmf$ to $K\power{q}$.
In \cite{Laures-2} Laures
studied the $K(1)$--local topological modular forms
at the prime $2$,
and discussed the relationship   
between the $q$--expansion map, Witten genus and 
$MO\langle 8\rangle$--orientation of $\tmf$. 

A generalization of Chern character 
to higher chromatic level
has been considered by Ando, Morava and Sadofsky~\cite{AMS}
under geometric background.
Their generalized Chern character 
is a multiplicative natural transformation
from $(n+1)$--th Morava $E$--theory $E_{n+1}$
to the $n$--th Morava $E$--theory with coefficients
in some big Cohen ring.
In \cite{Torii-5} we studied the degeneration
of formal group law,
which is used to construct their Chern character.
By using the results in \cite{Torii-5}
we refined their generalized Chern character in \cite{Torii-9}. 
Then we were able to control it algebraically.
In this note we reformulate and extend some results in \cite{Torii-9}.

Let ${\mathcal S}$ be the stable homotopy category 
of $p$--local spectra for some prime $p$.
It is known that there is a filtration of full subcategories
of ${\mathcal S}$, which corresponds to the height filtration
of the moduli space of one-dimensional commutative formal group laws; 
see Devinatz, Hopkins and Smith \cite{DHS}, Hopkins and Smith \cite{Hopkins-Smith},
Morava \cite{Morava} and Ravenel \cite{Ravenel-2}.
The layers of this filtration are equivalent to 
the $K(n)$--local categories, where $K(n)$ is the $n$--th Morava $K$--theory.
Hence it is considered that
the stable homotopy category ${\mathcal S}$ is built from
$K(n)$--local categories.
In fact, the chromatic convergence theorem (cf Ravenel \cite{Ravenel-2})
says that for a $p$--local finite spectrum $X$,
the natural tower 
$\cdots \to L_{n+1}X\to L_nX\to \cdots \to L_0X$ recovers
$X$, that is, 
the homotopy inverse limit of the tower is homotopy equivalent to $X$.
Furthermore,
the chromatic splitting conjecture (cf Hovey \cite{Hovey-1})
implies that the $p$--completion of a finite spectrum $X$
is a direct summand of the product $\prod_n L_{K(n)}X$.
This means that it is not necessarily to reconstruct 
the tower but it is sufficient to
know all $L_{K(n)}X$ to obtain some information about $X$.  
In \cite{Hovey-1} Hovey observed that
the weak form of the chromatic splitting conjecture should imply
many interesting results.
The weak form means that the canonical map
$L_n(S^0)_p^{\wedge}\to L_nL_{K(n+1)}S^0$ is a split monomorphism.
In \cite[Remark~3.1(i)]{Minami-1}
Minami indicated that the weak form implies that
there is a natural map $\rho_X$ for a finite spectrum $X$
from the $K(n+1)$--localization $L_{K(n+1)}X$ to the 
$K(n)$--localization $L_{K(n)}X$ such that 
the following diagram commutes:
\begin{equation}\label{topological-diagram}\begin{aligned}
\xymatrix{
 &  X \ar[dr]^{\eta_{K(n)}} \ar[dl]_{\eta_{K(n+1)}} &  \\
L_{K(n+1)}X \ar[rr]^{\rho} &&  L_{K(n)}X
}\end{aligned}
\end{equation}
where $\eta_K(n)\co X\to L_{K(n)}X$ and 
$\eta_K(n+1)\co X\to L_{K(n+1)}X$ are the localization maps.

In \cite{Torii-9} we considered the modulo $I_n$ version of the
algebraic analogue of the diagram~\eqref{topological-diagram}. 
The Morava $E$--theory $E_n$ defines a functor from the 
$K(n)$--local category to the category of twisted 
$E_{n*}$--$G_n$--modules, where $G_n$ is the $n$--th extended
Morava stabilizer group.
The Adams--Novikov spectral sequence based on $E_n$--theory
has its $E_2$--term $H^{**}_c(G_n; E_{n*}(X))$ and
converges to $\pi_*(L_{K(n)}X)$ strongly if $X$ is finite,
where $H^{**}_c(G_n;-)$ is the continuous cohomology group of $G_n$. 
Hence 
the category of twisted $E_{n*}$--$G_n$--modules
can be considered as an algebraic approximation of the $K(n)$--local category.

Let $BP_*$ be the Brown--Peterson spectrum at an odd prime $p$  
and $I_n$ the invariant prime ideal generated by
$p,v_1,\ldots,v_{n-1}$.
There is a commutative ring spectrum $E=E_{n+1}/I_n$,  
which is a complex oriented cohomology theory with 
coefficient ring $E_*=E_{n+1,*}/I_n=\F\power{u_n}[u^{\pm 1}]$.
We denote by $K^*(-)$ a variant of $n$--th Morava $K$--theory with
coefficient $\F[w^{\pm 1}]$.  
The modulo $I_n$--version of the algebraic analogue of 
the diagram~\eqref{topological-diagram} is 
\begin{equation}\label{algebraic-diagram-mod-In}\begin{aligned}
\xymatrix{
 &  {\mathcal S}^f \ar[dr]^{K} \ar[dl]_{E} &  \\
{\mathcal M}_{n+1}^f \ar[rr]^{\mu} && {\mathcal M}_n^f
}\end{aligned}
\end{equation}  
where $\mathcal{S}^f$ is the stable homotopy category of finite spectra,
${\mathcal M}_{n+1}^f$ is the category of 
finitely generated twisted $E_*$--$G_{n+1}$--modules 
and ${\mathcal M}_n$ is the category 
of finitely generated twisted $K_*$--$G_n$--modules. 
So the question is:
Does there exist a functor $\mu$ from ${\mathcal M}_{n+1}^f$ to
${\mathcal M}_n^f$ 
which makes the diagram~\eqref{algebraic-diagram-mod-In} commute?

In \cite{Torii-5} we constructed a Galois extension $L$
of the quotient field of $\F\power{u_n}$,
over which the formal group $F_{n+1}$ associated with
$E$ is nicely isomorphic to the Honda group law $H_n$.
By using this result,
we constructed a natural transformation
$\Theta$ from $E$--cohomology to 
$K$--cohomology with coefficients in $L$
in \cite{Torii-9}:
\[ \Theta\co E^*(X)\longrightarrow 
   L^*(X). \] 
This is regarded as a generalized Chern character
since it is a multiplicative natural transformation
from the height--$(n+1)$ cohomology $E$ to
the height--$n$ cohomology $K$ with coefficients in $L$. 
Then it is shown that $\Theta$ is equivariant with respect to
the action of 
$G_{n+1}$,
and $\Theta$ induces $L_* \otimes_{E_*} E^*(Z)\cong L_*\otimes_{K_*}K^*(Z)$,
a natural isomorphism of 
${\mathcal G}$--modules
for any finite spectrum $Z$,
where ${\mathcal G}=\Gamma\ltimes (S_n\times S_{n+1})$
and $L_*=L[u^{\pm 1}]$.
By these results,
we have shown that there is a natural isomorphism
of twisted $K_*$--$G_n$--modules:
\[ K^*(X)\cong \inverselimit{Z} H^0(S_{n+1}; L_* \otimes_{E_*} E^*(Z)) \]
for any CW-spectrum $X$ where
$Z$ ranges over finite subspectra of $X$
\cite[Corollary~4.3]{Torii-9}.
Hence if we set the functor $\mu(-)=H^0(S_{n+1}; L_*\otimes_{E_*}-)$, 
it makes the diagram~\eqref{algebraic-diagram-mod-In} commute.

Essentially, the twisted $K_*$--$G_n$--module structure
gives $K^*(X)$
the stable cohomology operations except for Milnor operations.
So in this note we would like to extend the above result 
in the form which includes the action of Milnor operations.  
Note that the twisted $K_*$--$G_n$--module structure on $K^*(X)$ 
with Milnor operations is equivalent to 
the $\kk$--comodule structure.
In this note we construct a symmetric monoidal functor
${\mathcal F}$ from the category of profinite
$\ee$--precomodules to the category of profinite $\kk$--comodules.
Roughly speaking,
a profinite $\ee$--precomodule is a filtered inverse limit
of finitely generated $E_*$--module 
$$M=\smash{\subrel{\longleftarrow}{\lim}} M/M_{\lambda}$$ such that
$M^c$ has a complete $\ee$--comodule structure,
where $$M^c=\subrel{\longleftarrow}{\lim}M/M_{\lambda}+{\mathfrak m}^iM.$$
For a profinite $\ee$--precomodule $M$,
there is a natural twisted $E_*$--$G_{n+1}$--module structure on $M$.
Hence $M\wtens_{E_*}L_*$ is a twisted $L_*$--$G_{n+1}$--module.
We set 
\[ {\mathcal F}(M)=H^0(S_{n+1};M\wtens_{E_*}L_*). \]

\begin{theoremA}[\fullref{cor:comodule2comodule}]
The functor ${\mathcal F}$ extends
to a symmetric monoidal functor from
the category of profinite $\ee$--precomodules
to the category of profinite $\kk$--comodules.  
\end{theoremA}

For a spectrum $X$ 
we denote by $\Lambda(X)$ the category
whose objects are maps $Z\stackrel{u}{\to}X$ with $Z$ finite.
We associate to $X$ a cofiltered system
$\mathbb{E}^*(X)=\{E^*(Z)\}$ and 
$\mathbb{K}^*(X)=\{K^*(Z)\}$ indexed by $\Lambda(X)$.
Then 
$E^*(X)=\ \inverselimit{}\mathbb{E}^*(X)$
and $K^*(X)=\ \inverselimit{}\mathbb{K}^*(X)$.
The following is the main theorem of this note.

\begin{theoremA}[\fullref{thm:Main_Theorem}]
For any spectrum $X$,
the generalized Chern character $\Theta$
induces a natural isomorphism of cofiltered
system of $\kk$--comodules:
\[ {\mathcal F}(\mathbb{E}^*(X))\cong \mathbb{K}^*(X). \]
If $X$ is a space, then this is an isomorphism 
of systems of $\kk$--comodule algebras. 
\end{theoremA}

The organization of this note is as follows.
In \fullref{subsection:Landweber_exact_over_pn}
we summarize well-known results on generalized cohomology theories 
which are Landweber exact over $\pn$.
In \fullref{subsection:review_degeneration}
we review our main result in \cite{Torii-5}
on degeneration of formal group laws.
In \fullref{subsection:GCh} we review on the construction of the generalized
Chern character.

In \fullref{subsection:complete-module}
we study the category of complete Hausdorff
filtered modules and the action of a profinite group on a complete module.
In \fullref{subsection:Complete_Hopf_algebroids}
we recall complete Hopf algebroids and their comodules.

In \fullref{subsection:complete_hopf_algebroid_C}
we describe the structure of complete Hopf algebroid $C(G,R^c)$,
where $C(G,R^c)$ is the ring of all continuous functions from 
a profinite group $G$ to an even-periodic complete local ring $R^c$.
In \fullref{subsection:twisted_modules} 
we show the well-known fact that the category of 
complete $C(G,R^c)$--comodules
is equivalent to the category of complete twisted $R^c$--$G$--modules.
Usually, we use and study the  category of complete 
twisted $\fp{n}$--$G_n$--modules.
In \fullref{subsection:Lemma on twisted modules}
we show that there is no essential difference between the 
category of complete twisted $\F$--$G_n$--modules and the category
of complete twisted $\fp{n}$--$G_n$--modules.
In \fullref{subsection:reformulation}
we reformulate a result of \cite{Torii-9}.
We construct a symmetric monoidal functor ${\mathcal F}$ 
from the category of profinite $C(G_{n+1},E^c_*)$--precomodules
to the category of profinite $C(G_n,K_*)$--comodules,
and show that there is a natural isomorphism between
${\mathcal F}(\mathbb{E}^*(X))$ and 
$\mathbb{K}^*(X)$ as systems of $C(G_n,K_*)$--comodules.

In \fullref{subsection:Co-operation ring} 
we define a complete co-operation ring $\acha$ for $A=E_{n+k}/I_n$,
and study a $\acha$-(pre)comodule algebra structure on 
the $A$--cohomology of the projective space $\cp$ and
the lens space $S^{2p^n-1}/(\zp{})$.
In \fullref{subsection:Lambda_E_Lambda_K}
we study a twisted $E_*$--$G_{n+1}$--module structure on 
the exterior algebra $\Lambda_{E_*}$ and show that
${\mathcal F}(\Lambda_{E_*})\cong \Lambda_{K_*}$ 
as twisted $K_*$--$G_n$--modules.
In \fullref{subsection:milnot_op}
we define Milnor operations $Q^A_i$ for a $\Lambda_{A_*}$--comodule $M$.
In \fullref{subsection:compatibility}
we study $\acha$--comodule structures in terms 
of $C_{A_*}$--comodule structures and $\Lambda_{A_*}$--comodule structures.
We show that an $\acha$--comodule structure is equivalent 
to a $C_{A_*}$--comodule structure and a $\Lambda_{A_*}$--comodule structure
which satisfy some compatibility condition.

In \fullref{subsection:symmetric_monoidal_dunctor}
we extend the symmetric monoidal functor ${\mathcal F}$
from the category of profinite $\ee$--precomodules to
the category of profinite $\kk$--comodules.
In \fullref{subsection:Main_Theorem}
we prove the main theorem.


In this note $p$ shall be an odd prime except for 
\fullref{subsection:review_degeneration},
$\F$ a finite field containing $\fp{n}$ and $\fp{n+1}$,
and $\Gal$ the Galois group $\gal{\F/\fp{}}$. 
We think a group $G$ acts on a ring $R$ from the right and 
denote by $r^g$ the right action of $g\in G$ on $r\in R$.
For a power series $\alpha(X)=\sum \alpha_i X^i\in R\power{X}$,
we set $\alpha^g(X)=\sum \alpha_i^g X^i$ for $g\in G$.
An $R$--module means a left $R$--module if nothing else is indicated. 

\section{The generalized Chern character}
\label{section:review_degeneration}

\subsection[Landweber exact theories over P(n)]{Landweber exact theories over $\pn$}
\label{subsection:Landweber_exact_over_pn}

\begin{definition}\rm
A graded commutative ring $R_*$ is said to be even-periodic 
if $R_*$ is concentrated in even degrees and 
$R_2$ contains a unit in $R_*$.
A multiplicative generalized cohomology theory
$h^*(-)$ is said to be even-periodic if 
the coefficient ring $h_*=h^*(\pt)$ is even-periodic.
\end{definition}

For a spectrum $X$,
we denote by $\Lambda(X)$ the category
whose objects are maps $Z\stackrel{u}{\to}X$ 
such that $Z$ is finite,
and whose morphisms are maps $\smash{Z\stackrel{v}{\smash{\to}\vphantom{-}}Z'}$
such that $u'v=u$.
Then $\Lambda(X)$ is an essentially small filtered category.

\begin{definition}\rm
Let $h^*(-)$ be a generalized cohomology theory.
For a spectrum X, we define a filtration on $h^*(X)$ 
indexed by $\Lambda(X)$ as 
\[ F^Z h^*(X)=\mbox{\rm Ker}(h^*(X)\longrightarrow h^*(Z)) \]
for $Z\in \Lambda(X)$.
We call this filtration the profinite filtration
and the resulting topology the profinite topology.
\end{definition}

\begin{remark}\label{remark:profinite_topology}\rm
If $h^*(-)$ is even-periodic and the degree--$0$ 
coefficient ring $h_0$ is a complete
Noetherian local ring, 
then $h^*(Z)$ is a finitely generated $h_*$--module for all 
$Z\in \Lambda(X)$, 
and the canonical homomorphism
\[ h^*(X)\longrightarrow
   \inverselimit{Z\in \Lambda(X)}h^*(Z)\]
is an isomorphism.
This implies that 
$h^*(X)$ is complete Hausdorff with respect to the profinite topology.
\end{remark}

Let $BP$ be the Brown--Peterson spectrum at an odd prime $p$,
whose coefficient ring is given by 
$BP_*={\mathbf Z}_{(p)}[v_1,v_2,\ldots]$ with $|v_i|=2(p^i-1)$.
Let $I_n$ be the invariant prime ideal generated by
$p,v_1,\ldots,v_{n-1}$.
There is a commutative $BP$--algebra spectrum $\pn$,
whose coefficient ring is
$\pn_*=BP_*/I_n$.
In particular, $P(0)=BP$.
Let ${\mathcal X}={\mathbf C}P^{\infty}$ the infinite dimensional
complex projective space,
and ${\mathcal Y}= S^{2p^n-1}/({\mathbf Z}/p)$ 
the lens space of dimension $2p^n-1$,
where ${\mathbf Z}/p$ is the cyclic group of order $p$ acting 
on the unit sphere $S^{2p^n-1}$ in ${\mathbf C}^{p^n}$ by standard way.
These spaces are important test spaces 
to stable cohomology operations of 
complex oriented cohomology theories (cf \cite[Section 14]{Boardman}).
The $P(n)$--cohomology of ${\mathcal X}$ and ${\mathcal Y}$ are
given as follows:
\[ \begin{array}{rcl}
    P(n)^*({\mathcal X})&=&P(n)_*\power{x},\\[1mm]
    P(n)^*({\mathcal Y})&=&\Lambda(y)\otimes P(n)_*[x]/(x^{p^n}),\\   
   \end{array}\]
where $x\in P(n)^2({\mathcal X})$ is the orientation class
and $y\in P(n)^1({\mathcal Y})$.

Let $F$ be a $p$--typical formal group law over a commutative ring $R$.
By universality of the $p$--typical formal group law $F_{BP}$
associated to $BP$,
there is a unique ring homomorphism $f\co BP_*\to R$
such that $F$ is the base change of $F_{BP}$ by $f$.
If $f(v_i)=0$ for $0\le i<n$,
then $f$ induces a ring homomorphism 
$\overline{f}\co\pn_*\to R$.
In this case we say that a $p$--typical formal group law $F$
is of strict height at least $n$.
Hence $\pn_*$ is the universal ring of $p$--typical
formal group law of strict height at least $n$.
We say that a ring homomorphism $\pn_*\to R$
is Landweber exact over $\pn$,
if the sequence $v_n,v_{n+1},\ldots$ is regular in $R$.
In this case, the functor $R_*\otimes_{\pn_*}\pn_*(-)$
is a generalized homology theory by
Landweber--Yagita exact functor theorem \cite{Landweber,Yagita},
where $R_*$ is the even-periodic commutative ring
$R[u^{\pm 1}]$ with $|u|=-2$.
Furthermore, if $R$ is a complete Noetherian local ring,
then 
\[ R^*(X)=\inverselimit{\Lambda(X)} (R_*\otimes_{\pn_*}\pn^*(Z))\]
is a generalized cohomology theory.

Let $\F$ be a finite field 
which contains the finite fields $\fp{n}$ and $\fp{n+1}$.
Let $E_n^*(-)$ be a variant of Morava $E$--theory whose coefficient ring
is given by
\[ E_{n*}=W(\F)\power{u_1,\ldots,u_{n-1}}[u^{\pm 1}],\]
where $W(\F)$ is the ring of Witt vectors with coefficients in $\F$.
The grading is given by $u_i=0$ for $1\le i<n$ and $|u|=-2$.
Then the degree--$0$ formal group law $F_n$ associated to $E_n$ 
is a universal deformation of the Honda group law $H_n$
of height $n$ over $\F$.
For $0\le k\le n$,
there is a commutative multiplicative cohomology theory
$(E_n/I_k)^*(-)$ whose coefficient ring is just $E_{n*}/I_k$,
where $I_k$ is the invariant prime ideal $(p,u_1,\ldots,u_{k-1})$.

We define even-periodic graded commutative rings $E_*$ and $K_*$
as follows:
\[ \begin{array}{rcl}
      E_* & = & \F\power{u_n}[u^{\pm 1}],\\[1mm]
      K_* & = & \F[w^{\pm 1}],\\ 
   \end{array}\]
where the gradings are given by $|u_n|=0, |w|=|u|=-2$.
The ring homomorphisms $P(n)_*\to E_*$ 
given by $\smash{v_n\mapsto u_n u^{-(p^n-1)}, 
v_{n+1}\mapsto u^{-(p^{n+1}-1)}, v_i\mapsto 0}$ for $i>n+1$,
and $P(n)_*\to K_*$
given by
$\smash{v_n\mapsto w^{-(p^n-1)}, v_i\mapsto 0}$ for $i>n$,
make $E_*$ and $K_*$ Landweber exact $P(n)_*$--algebras,
respectively.
Hence 
\[ \begin{array}{rcl}
    E^*(X)&=& \inverselimit{} 
    (E_*{\otimes}_{P(n)_*}P(n)^*(Z)),\\[1mm]
    K^*(X)&=& \inverselimit{}
    (K_*{\otimes}_{P(n)_*}P(n)^*(Z)),\\
   \end{array}\]
define generalized cohomology theories. 
Note that there are no limit one problems since
the degree--$0$ subrings are complete Noetherian local rings, respectively.
The cohomology theory $K^*(-)$ is a variant of Morava $K$--theory
and the associated degree--$0$ formal group law
is the Honda group law $H_n$ of height $n$ over $\F$.
Since the cohomology theory $E^*(-)$ is $(E_{n+1}/I_n)^*(-)$,
the associated degree--$0$ formal group law
is the base change of $F_{n+1}$ to $\F\power{u_n}$. 
\begin{align*}
     x_E&= 1\otimes x \in  E^0({\mathcal X}),\tag*{\hbox{We set}}\\
     y_E&= 1\otimes y \in  E^1({\mathcal Y}),\\ 
     x_K&= 1\otimes x \in  K^0({\mathcal X}),\\
     y_K&= 1\otimes y \in  K^1({\mathcal Y}). 
\end{align*}
\begin{align*}
    E^*({\mathcal X})& \cong  E_*\power{x_E},\tag*{\hbox{Then we have}}\\
    E^*({\mathcal Y})& \cong  
    \Lambda(y_E)\otimes E_*[x_E]/(x_E^{p^n}),\\
    K^*({\mathcal X})& \cong  K_*\power{x_K},\\
    K^*({\mathcal Y})& \cong  \Lambda(y_K)\otimes K_*[x_K]/(x_K^{p^n}).
\end{align*}

\subsection{Degeneration of formal groups}
\label{subsection:review_degeneration}

In this subsection we review some results in \cite{Torii-5}.
In this subsection $p$ is any prime number.
Let $E_{n+1,0}$ be the degree--$0$ coefficient ring of 
the variant of Morava $E$--theory $E_{n+1}$:
\[ E_{n+1,0}= W(\F)\power{u_1,\ldots,u_n}.\]
The associated degree--$0$ formal group law $F_{n+1}$ 
is a universal deformation of the Honda group laws $H_{n+1}$ of height $n+1$
over $\F$.
The extended Morava stabilizer group $G_{n+1}=\Gal\ltimes S_{n+1}$
is the automorphism group of $F_{n+1}$ in some generalized sense
(cf Strickland \cite{Strickland} and Torii \cite{Torii-5}),
where $\Gal$ is the Galois group $\gal{\F/\fp{}}$
and $S_{n+1}$ is the $n$--th Morava stabilizer group.
Note that $S_{n+1}$ is the automorphism group of $H_{n+1}$
in the usual sense.
The extended Morava stabilizer group
$G_{n+1}$ is a profinite group and acts on $E_{n+1,0}$ continuously,
where the topology of $E_{n+1,0}$ is given by the adic-topology.
Since the ideal $I_n=(p,u_1,\ldots,u_{n-1})$ of $E_{n+1,0}$
is stable under the action,
$G_{n+1}$ also acts on the quotient ring $E_{n+1,0}/I_n=\F\power{u_n}$
continuously.

We regard the formal group law $F_{n+1}$ as being defined
over $\F\power{u_n}$ by the obvious base change.
This situation is a kind of degeneration and 
a fundamental technique to study a degeneration
is to investigate the monodromy representation.
Let $M=\F((u_n))$ be the quotient field of $\F\power{u_n}$
and $M^{\rm sep}$ its separable closure.
Then the height of $F_{n+1}$ on $M$ is $n$.
Hence the fibre of $F_{n+1}$ over $M^{\rm sep}$ is isomorphic
to $H_n$ since the isomorphism classes of 
formal group laws over a separably closed field 
are classified by their height.   
The monodromy representation of $F_{n+1}$ around the closed point
gives the following homomorphism:
\[ \gal{M^{\rm sep}/M}=\pi_1(M)\longrightarrow \mbox{\rm Aut}(H_n)=S_n. \]
This homomorphism was studied by Gross in \cite{Gross}.

Let $\Phi$ be an isomorphism over $M^{\rm sep}$
between $F_{n+1}$ and $H_n$:
\[ \Phi(F_{n+1}(X,Y))=H_n(\Phi(X),\Phi(Y)).\]
Let $L$ be a separable algebraic extension of $M$
obtained by adjoining all the coefficients of $\Phi(X)$.
Then the above homomorphism 
$\gal{M^{\rm sep}/M}\to S_n$ induces an isomorphism
$\smash{\gal{L/M}\stackrel{\cong}{\to}S_n}$,
and this extends to an isomorphism 
$\gal{L/\fp{}((u_n))}\cong G_n$.
Let ${\mathcal G}$ be the semidirect product
$\Gal\ltimes (S_n\times S_{n+1})$.
Then ${\mathcal G}$ is a profinite group,
and contains $G_n$ and $G_{n+1}$ as closed subgroups.

The following theorem is a main point of \cite{Torii-5}.

\begin{theorem}{\rm \cite[Section 2.4]{Torii-5}}\qua
\label{thm:degeneration_main}
The profinite group ${\mathcal G}$ acts on 
the formal group law $(F_{n+1},L)$ in generalized sense.
The action of the subgroup $G_{n+1}$ is an extension 
of the action on $(F_{n+1},\F\power{u_n})$.
The action of the subgroup $G_n$ on $(F_{n+1},L)$ 
is the action of Galois group on $L$ and the trivial action on $F_{n+1}$.
Under the isomorphism $\smash{\Phi\co F_{n+1}\stackrel{\cong}{\to}H_n}$,
the induced action of ${\mathcal G}$ on $(H_n,L)$ is encoded
as the following two commutative diagrams. 
For $g\in G_{n+1}$, there is a commutative diagram
\begin{equation}\label{eq:commutative_diagram_g_n+1_degeneration}
\begin{aligned}
\xymatrix{ F_{n+1} \ar[r]^{t_E(g)}\ar[d]_{\Phi} & F_{n+1}{}^g \ar[d]^{\Phi^g} \\
H_n \ar[r]^{=} & H_n^g,}
\end{aligned}
\end{equation}
where $t_E(g)(X)$ is an isomorphism from $F_{n+1}$ to $F_{n+1}{}^g$
corresponding to $g$.
For $h\in G_n$, there is a commutative diagram
\begin{equation}\label{eq:commutative_diagram_g_n_degeneration}
\begin{aligned}
\xymatrix{ F_{n+1} \ar[r]^{=}\ar[d]_{\Phi} & F_{n+1}{}^h \ar[d]^{\Phi^h} \\
H_n \ar[r]^{t_K(h)} & H_n^h,}
\end{aligned}
\end{equation}
where $t_K(h)(X)$ is the automorphism of $H_n$
corresponding to $h$.
\end{theorem}

\subsection{The generalized Chern character}
\label{subsection:GCh}

In this subsection we review the generalized Chern character
$\Theta$ constructed in \cite{Torii-9}. 

The co-operation ring $P(n)_*(P(n))$ is isomorphic to
\[ P(n)_*[t_1,t_2,\ldots]\otimes\Lambda(a_{(0)},\ldots,a_{(n-1)}) \]
as a left $P(n)_*$--algebra,
where $|t_i|=2(p^i-1)$ and $|a_{(i)}|=2p^i-1$.
In particular, $P(n)_*(P(n))$ is a free $P(n)_*$--module.
Hence $(P(n)_*,P(n)_*(P(n)))$ is a Hopf algebroid over $\fp{}$.
By formalism of Boardman \cite{Boardman},
there is a natural $P(n)_*(P(n))$--comodule structure on 
the completion $P(n)^*(X)^{\wedge}$ with respect to the
profinite topology:
\[ \rho\co P(n)^*(X)^{\wedge}
   \longrightarrow P(n)_*(P(n))\widehat{\otimes}_{P(n)_*}
         P(n)^*(X)^{\wedge}.\]
The set of $\fp{}$--algebra homomorphisms from
$P(n)_*(P(n))/(a_{(0)},\ldots,a_{(n-1)})$ to 
an even-periodic $\fp{}$--algebra $R_*$
is naturally identified with the set of triples $(F,f,G)$,
where $F$ and $G$ are $p$--typical formal group laws over $R_0$
with strict height at least $n$,
and $f$ is an isomorphism between them.
Let $L_*$ be an even-periodic $E_*$--algebra $L[u^{\pm 1}]$.
By \fullref{thm:degeneration_main} 
and the above moduli interpretation of the ring 
$P(n)_*(P(n))/(a_{(0)},\ldots,a_{(n-1)})$,
there is a ring homomorphism
$\theta\co P(n)_*(P(n))\to P(n)_*(P(n))/(a_{(0)},\ldots,a_{(n-1)})\to 
L_*$
such that the following diagram commutes:
\[
\xymatrix@C=30pt{
        P(n)_* \ar[r]^{H_n}\ar[d]_{\eta_R}  & K_* \ar[d]\\
    P(n)_*(P(n)) \ar[r]^-{\theta} & L_*\\
    P(n)_* \ar[r]^{F_{n+1}} \ar[u]_{\eta_L} & E_*. \ar[u]
    }
\]
That is, $\theta$ corresponds to the triple $(F_{n+1},\Phi,H_n)$ over $L$.
For $Z\in \Lambda(X)$,
by extending the natural ring homomorphism
\[ \begin{array}{rcl}
   P(n)^*(Z)&\stackrel{\rho}{\longrightarrow}&
   P(n)_*(P(n)){\otimes}_{P(n)_*}P(n)^*(Z)\\
   &\stackrel{\theta\otimes 1}{\longrightarrow}&
   L_*{\otimes}_{P(n)_*}P(n)^*(Z)\\[1mm]
   &\cong&   L^*(Z) 
   \end{array}\]
to $E^*(Z)=E_*\otimes_{P(n)_*}P(n)^*(Z)\to L^*(Z)$,
we obtain a multiplicative natural transformation
\begin{equation}\label{eq:generalized_Chern_character}
\Theta: E^*(X)\longrightarrow L^*(X),
\end{equation}
which we call the generalized Chern character.

The following lemma is easily checked.

\begin{lemma}\label{lemma:relations_E_K_fundamental_elements}
$\Theta(x_E)=\Phi(x_K)$ and $\Theta(y_E)=1\otimes y_K$.
\end{lemma}

\section{Complete Hopf algebroids}
\label{section:Complete_modules}

\subsection{Complete modules}\label{subsection:complete-module}

Let $k$ be a commutative ring.
We say that $(M,\{F^{\lambda}M\}_{{\lambda}\in \Lambda})$ 
is a filtered $k$--module 
if $M$ is an $k$--module and 
$\{F^{\lambda}M\}_{\lambda\in \Lambda}$ is a family of $k$--submodules
indexed by 
a (small) filtered category $\Lambda$.
Then $M$ can be given a linear topology.
We denote by $\fmod{k}$ the category of filtered $k$--modules
and continuous homomorphisms. 
A filtered $k$--module 
$(M,\{F^{\lambda}M\}_{{\lambda}\in \Lambda})$ is said to be
complete Hausdorff if
the canonical homomorphism $$M\to \subrel{\longleftarrow}{\lim}_{\Lambda} 
M/F^{\lambda}M$$
is an isomorphism.
We denote by $\fmodc{k}$ the full subcategory of $\fmod{k}$
whose objects are complete Hausdorff.
We say that $(M,\{F^{\lambda}M\}_{{\lambda}\in \Lambda})\in\fmodc{k}$ is 
a profinite $k$--module if 
$M/F^{\lambda}M$ is a finitely generated $k$--module for all 
$\lambda\in \Lambda$.
We denote by $\profg{k}$ the full subcategory
of $\fmodc{k}$ whose objects are profinite.

Since $\fmodc{k}$ is a symmetric monoidal category with
tensor product $\wtens_k$ and unit object $k$,
we can define commutative monoid objects in $\fmodc{k}$,
that is, complete commutative $k$--algebras.
We denote by $\falgc{k}$ the category of complete commutative $k$--algebras.
For $R\in \falgc{k}$,
we can define an $R$--module in $\fmodc{k}$,
and we denote by $\fmodc{R}$ the category of $R$--modules.
For $R_1\to R_2\in\falgc{k}$,
there is a base change functor 
\[ (-)\widehat{\otimes}_{R_1} R_2\co
   \fmodc{R_1}\ {\hbox to 10mm{\rightarrowfill}}\ \fmodc{R_2}.\]
If $R$ is a complete Noetherian local $k$--algebra
with maximal ideal ${\mathfrak m}$,
then $R$ with ${\mathfrak m}$--adic filtration
can be regarded as an object in $\falgc{k}$.  
We denote by $R^c$ the $k$--module $R$ with 
${\mathfrak m}$--adic filtration,
and simply by $R$ the $k$--module $R$ with trivial filtration $\{0\}$. 
Note that the base change $M\widehat{\otimes}_R R^c$ 
for $M\in \fmodc{R}$ is given by
$$\smash{\subrel{\longleftarrow}{\lim}} M/(F^{\lambda}M+{\mathfrak m}^iM)$$
with inverse limit topology.
Since $M$ is isomorphic to the inverse limit of $M/{\mathfrak m}^iM$
for a finitely generated $R$--module $M$,
we see that $M\widehat{\otimes}_R R^c\cong M$ as 
(abstract) $R$--modules for $M\in \profg{R}$.

\begin{example}
Let $h^*(-)$ be a generalized cohomology theory and $X$ a spectrum. 
We defined the profinite filtration on $h^*(X)$
in \fullref{remark:profinite_topology}. 
If $h^*(-)$ is even-periodic and the degree--$0$ coefficient ring
$h_0$ is a complete Noetherian local ring,
then $h^*(X)$ is a complete Hausdorff profinite $h_*$--module.
Hence the cohomology theory $h^*(-)$
gives a functor from the stable homotopy category
to $\profg{h_*}$.
\end{example}

\begin{lemma}
Let $M\in \fmodc{k}$ and $\underline{M}$ the underlying $k$--module.
If $M\in \fmodc{R}$,
then $\underline{M}$ is an $R$--module in the usual sense.
\end{lemma}

\begin{proof}
The map $R\otimes M\to R\wtens M\to M$ gives
an $R$--module structure on $\underline{M}$.
\end{proof}

\begin{lemma}
If $M\in\fmodc{R}$,
then for any open $k$--submodule $M_{\lambda}$
there is an open $R$--submodule $N$
such that $N\subset M_{\lambda}$.
\end{lemma}

\begin{proof}
The fact that $M\in \fmodc{R}$
implies that 
the map $R\wtens M\to M\to M/M_{\lambda}$ 
factors through $R\otimes M/M_{\mu}$ for some open $k$--submodule $M_{\mu}$.
Hence $R\cdot M_{\mu}\subset M_{\lambda}$.
We take $N$ as $R\cdot M_{\mu}$.
Since $M_{\mu}\subset N$, $N$ is an open $R$--submodule.
This completes the proof.    
\end{proof}

\begin{corollary}
Let $M\in \fmodc{R}$.
There is a fundamental (open) neighborhood system 
at $0$ consisting of $R$--submodules.
\end{corollary}

\begin{corollary}\label{cor:description_tensor}
Let $M$ and $N$ be objects in $\fmodc{R}$.
Then 
\[ M\subrel{R}{\wtens}N\cong\ \inverselimit{\lambda,\mu}
   (M/F^{\lambda}M)\otimes_R (N/F^{\mu}N),\]
where $\{F^{\lambda}M\}_{\lambda}$ and $\{F^{\mu}N\}_{\mu}$
are families of all open $R$--submodules
of $M$ and $N$, respectively.
\end{corollary}


\begin{lemma}\label{lemma:compatibility_d_c}
Let $\smash{M\in\fmodc{R}}$.
Then $M$ is an $R^c$--module compatible with given $R$--module structure
if and only if
for any open $R$--submodule $N$ 
there is a nonnegative integer
$i$ such that ${\mathfrak m}^iM\subset N$.
\end{lemma}

\begin{proof}
If $M$ is an $R^c$--module compatible with given $R$--module structure,
then there is a continuous map
$R^c\wtens_R M\longrightarrow M$,
which makes $M$ an $R^c$--module.  
Then the map $R^c\wtens_R M\to M\to M/N$
factors through $R/{\mathfrak m}^i\otimes_R M/N'$
for some $i$ and some open $R$--submodule $N'$.
This implies that ${\mathfrak m}^iM\subset 
{\mathfrak m}^iM+N'\subset N$.

If for any open $R$--submodule $N$
there is $i$ such that ${\mathfrak m}^iM\subset N$,
then there are compatible maps
$R^c\wtens_R M\to R/{\mathfrak m}^i \otimes_R M/N\to M/N$,
which induce a continuous map
$R^c\wtens_R M\to M$.
This map defines an $R^c$--module structure on $M$
compatible with given $R$--module structure.
\end{proof}

%

\begin{lemma}\label{lemma:profinite_action}
Let $M\in \fmodc{k}$.
If a profinite group $G$ acts on $M$ continuously
as $k$--module homomorphisms,
then for any open submodule $M_{\lambda}$
there is an open submodule $N$ such that $G\cdot N\subset M_{\lambda}$.
\end{lemma}

\begin{proof}
For any $g\in G$,
there are an open submodule $N_g$ of $M$
and an open neighborhood $U_g$ of $g$
such that $U_g\cdot N_g\subset M_{\lambda}$.
Since $G$ is compact,
$G=U_{g_1}\cup \cdots \cup U_{g_n}$.
Take an open submodule $N$ 
such that $N \subset N_{g_1}\cap \cdots\cap N_{g_n}$. 
Then for any $g\in G$, $g\in U_{g_i}$ for some $i$, and 
for any $x\in N\subset N_{g_i}$,
$g\cdot x\in U_{g_i}\cdot N_{g_i}\subset M_{\lambda}$.
Hence we obtain that $G\cdot N\subset M_{\lambda}$.
\end{proof}

\begin{corollary}\label{cor:lemma:profinite_action}
Let $M\in\fmodc{k}$.
If a profinite group $G$ acts on $M$ continuously
as $k$--modules homomorphisms,
then for any open submodule $M_{\lambda}$,
there is an open $G$--submodule $N$
such that $N\subset M_{\lambda}$.  
\end{corollary}

\begin{proof}
By \fullref{lemma:profinite_action},
there is an open submodule $N'$ such that $G\cdot N'\subset M_{\lambda}$.
Let $N$ be the submodule generated by $G\cdot N'$.
Then $N$ is a $G$--submodule and $N \subset M_{\lambda}$. 
Since $N'\subset N$,
$N$ is an open submodule. 
This completes the proof.
\end{proof}

\begin{corollary}
Let $M\in \fmodc{k}$ and $G$ a profinite group.
Suppose that $G$ acts on $M$ continuously
as $k$--module homomorphisms.  
There is a fundamental (open) neighborhood system 
at $0$ consisting of $G$--submodules.
\end{corollary}

\begin{theorem}
Let $G$ be a profinite group acting on $R^c$ continuously
as $k$--algebra homomorphisms,
and $M$ a complete twisted $R^c$--$G$--module.
For any open $R$--submodule $M_{\lambda}$, 
there is an open $R$--$G$--submodule $N$ such that $N\subset M_{\lambda}$. 
\end{theorem}

\begin{proof}
By \fullref{cor:lemma:profinite_action},
there is an open $k$--$G$--submodule $N'$
such that $N'\subset M_{\lambda}$.
Let $N$ be the $R$--submodule generated by $N'$.
Then $N$ is an open $G$--submodule such that $N\subset M_{\lambda}$.
\end{proof}

\begin{corollary}\label{cor:fund_neigh_twisted_mod}
Let $M$ be a complete twisted $R^c$--$G$--module.
Then there is a fundamental (open) neighborhood system 
at $0$ consisting of $R$--$G$--submodules.
\end{corollary}

\subsection{Complete Hopf algebroids and complete precomodules}
\label{subsection:Complete_Hopf_algebroids}

Let $A$ and $\Gamma$ be objects in $\falgc{k}$.
We suppose that there are maps in $\falgc{k}$:
\[ \begin{array}{rrcl}
    \eta_R\co & A &\longrightarrow & \Gamma,\\
    \eta_L\co & A &\longrightarrow & \Gamma,\\
    \chi  \co & \Gamma &\longrightarrow & \Gamma,\\
    \varepsilon\co & \Gamma & \longrightarrow & A.\\
   \end{array}\]
If the maps $(\eta_R,\eta_L,\chi,\varepsilon)$
satisfy the usual Hopf algebroid relations \cite[Appendix~1]{Ravenel}, 
then we say that the pair $(A,\Gamma)$ is a complete Hopf algebroid
over $k$.
A Hopf algebroid is a complete Hopf algebroid with discrete topology.
Since $\pn_*(\pn)$ is free over $\pn_*$,
$\pn_*(\pn)$ is a Hopf algebroid,
hence, a complete Hopf algebroid over $\fp{}$.

Let $A\to B$ be a map in $\falgc{k}$.
We set 
\[ \Gamma_B:= B\wtens_A \Gamma \wtens_A B .\]
Then $(B,\Gamma_B)$ is a complete Hopf algebroid over $k$
as usual.

\begin{example}\rm
Let $h^*(-)$ be an even-periodic Landweber exact theory 
over $\pn$ such that $h_0$ is a complete Noetherian local ring.
Put 
\[ \Gamma(h)=h_*^c\wtens_{\pn_*}\pn_*(\pn)\wtens_{\pn_*} h_*^c. \]
Then
$(h_*^c,\Gamma(h))$ is a complete Hopf algebroid over 
$\fp{}$.
\end{example}

An object $M\in \fmodc{A}$ is said to be a complete
$\Gamma$--comodule if there is a continuous map
$\rho\co M\to \Gamma\wtens_{A}M$ such that
obvious co-associativity and co-unity diagrams commute.

\begin{definition}\rm

Let $R$ be a complete Noetherian local ring,
and $(R^c,\Gamma)$ a complete Hopf algebroid over $k$.
An object $M\in \fmodc{R}$ 
is said to be a complete $\Gamma$--precomodule if
there is a continuous map
\[ \rho\co  M\longrightarrow \Gamma\wtens_R M \] 
such that the following two conditions are satisfied:
\begin{enumerate}
\item
For any open $R$--submodule $M_{\lambda}$ of $M$,
there is an open $R$--submodule $M_{\mu}$ such that 
the map $\smash{M\stackrel{\rho}{\smash{\to}\vphantom{-}}\Gamma\wtens_R M\to 
\Gamma\wtens_R M/M_{\lambda}}$
factors through $M/M_{\mu}$.
When the above condition is satisfied, 
$\rho$ induces a continuous map 
$\rho^c\co M^c\to \Gamma\wtens_{R^c}M^c$.

\item The continuous map $\rho^c$ makes $M^c$
a complete $\Gamma$--comodule.
\end{enumerate}  

Furthermore,
if $M$ is a complete Hausdorff commutative $R$--algebra and
$\rho^c$ is a map of complete $R^c$--algebras,
then $M$ is said to be a complete 
$\Gamma$--precomodule algebra.      
\end{definition}

For $Z$ finite,
the coaction map 
$\pn^*(Z)\to \pn_*(\pn)\otimes_{\pn_*}\pn^*(Z)$ 
induces a natural continuous map
\[ h^*(Z)\longrightarrow \Gamma(h)\wtens_{h_*}h^*(Z). \]

\begin{proposition}\label{prop:precomdule_cohomology}\rm
Let $h^*(-)$ be an even-periodic Landweber exact theory 
over $\pn$ such that $h_0$ is a complete Noetherian local ring.
Then $h^*(Z)$ has a natural $\Gamma(h)$--precomodule structure
for finite $Z$.
Furthermore, 
if $Z$ is a finite CW-complex,
then $h^*(Z)$ is a $\Gamma(h)$--precomodule algebra. 
\end{proposition}

\begin{proof}
Since $h^*(Z)$ is discrete if $Z$ is finite,
the condition $(1)$ is trivial.
Actually,
$h^*(Z)$ is a $h_*\otimes_{P(n)_*}P(n)_*(P(n))
\otimes_{P(n)_*}h_*$--comodule.
Hence $h^*(Z)^c=h^c_*\widehat\otimes_{h_*}h^*(Z)$
is $\Gamma(h)$--comodule.
If $Z$ is a finite CW-complex.
it is easy to see that $h^*(Z)$ is a $\Gamma(h)$--precomodule algebra.
\end{proof}

\section[Complete Hopf algebroid of continuous maps from G to the m-adic topological ring R*]{Complete Hopf algebroid $C(G,R^c_*)$}
\label{section:Complete_Hopf_algebroids}

\subsection[The Hopf algebroid structure of C(G,Rc*)]{The Hopf algebroid
structure of $C(G,R^c_*)$}
\label{subsection:complete_hopf_algebroid_C}

Let $k$ be a commutative ring and 
$R_*$ an even-periodic graded commutative $k$--algebra such that
the degree--$0$ subring $R_0$ is 
a complete local ring with maximal ideal ${\mathfrak m}_0$.
We denote by $R^c_*$ a graded topological ring $R_*$
with ${\mathfrak m}$--adic topology,
where ${\mathfrak m}={\mathfrak m}_0 R$.
Let $G$ be a profinite group,
which continuously acts on $R^c_*$ 
as $k$--algebra automorphisms from the right.
Let $C=C(G,R^c_*)$ be the set of all continuous maps from $G$ to $R^c_*$.
Then $C$ is an even-periodic commutative ring 
from the ring structure on $R^c_*$.
It is known that the pair $(R^c_*,C)$ is 
a graded complete Hopf algebroid over $k$.
In this section we 
describe the structure of 
$(R^c_*,C)$ (cf \cite[Section 6.3]{Hovey}).

First, note that there is an isomorphism of commutative rings 
\[ C=C(G,R^c_*)\cong \inverselimit{\smash[t]{i}}C(G,R_*/{\mathfrak m}^i),\]
where $C(G,R_*/{\mathfrak m}^i)$ is the ring of all continuous
map from $G$ to the discrete ring $R_*/{\mathfrak m}^i$.
We give the inverse limit topology to $C$,
where $C(G,R_*/{\mathfrak m}^i)$ is discrete.   
The projection $R^c_*\times G\to R^c_*$ gives a continuous ring homomorphism
$\eta_R\co R^c_*\to C$.
By the ring homomorphism $k\to R^c_*\smash{\stackrel{\eta_R}{\to}}C$,
we regard $C$ as a commutative $k$--algebra.  
The action $R^c_*\times G\to R^c_*$ gives a continuous ring homomorphism
$\eta_R\co R^c_*\to C$,
which is a $k$--algebra homomorphism.

Let $C(G\times G,R^c_*)$ be the ring of all continuous maps
from $G\times G$ to $R^c_*$.
Then $C(G\times G, R^c_*)$ is a complete commutative $k$--algebra as in $C$.

Let $G$ be a profinite group.
We denote by $C(G,M)$ the set of all continuous maps from $G$ to 
$M\in\fmodc{k}$.
Then it can be given a $k$--module structure on $C(G,M)$ from
the $k$--module structure on $M$.  
There is an isomorphism of $k$--modules
\[ C(G,M)\cong\ \inverselimit{N}\,\directlimit{U}
              F(G/U,M/N),\]
where $F(G/U,M/N)$ is the set of all maps from $G/U$ to $M/N$,
$N$ ranges over all open submodules of $M$,
and $U$ ranges over all open normal subgroup of $G$.
We regard $C(G,M)$ as an object in $\fmodc{k}$
by inverse limit topology.

\begin{lemma}\label{lemma:complete_arguments_change}
For a profinite group $G$ and $M\in\fmodc{k}$,
there is a natural isomorphism in $\fmodc{k}$: 
\[  C(G,k)\wtens_k M \cong C(G,M).\] 
\end{lemma}

\begin{proof}
We have an isomorphism
$$C(G,k)\wtens M 
\cong\ \subrel{\longleftarrow}{\lim}\,\subrel{\longrightarrow}{\lim}
F(G/U,k)\otimes M/N.$$
Since $G/U$ is a finite set,
$F(G/U,k)\otimes M/N\cong F(G/U,M/N)$.
Hence we see that $C(G,k)\wtens M\cong\
\subrel{\longleftarrow}{\lim}\subrel{\longrightarrow}{\lim}
F(G/U,M/N)\cong C(G,M)$.
\end{proof}

Let $m\co C\times C\to C(G\times G,R^c_*)$ be 
a map given by
$m(\alpha,\beta)(g_1,g_2)=\alpha(g_1)^{g_2}\beta(g_2)$
for $\alpha,\beta\in C, g_1,g_2\in G$.
%
The map $m$ induces an isomorphism of complete commutative $k$--algebras:
\[ C\widehat{\subrel{R^c_*}{\otimes}}C\stackrel{\cong}{\longrightarrow}
   C(G\times G,R^c_*). \]
%
%
We define a map $\psi$ by
\[ \psi\co C\stackrel{\widetilde{\psi}}{\to}C(G\times G,R^c_*)
         \cong C\widehat{\subrel{R^c_*}{\otimes}}C, \]
where $\widetilde{\psi}$ is the map induced by the multiplication  
$G\times G\to G$.  
Then we can check that $\psi$ is a continuous $k$--algebra homomorphism.

Let $\chi\co C\to C$ be the map given by
$\chi(\alpha)(g)=\alpha(g^{-1})^g$ for $\alpha\in C, g\in G$.
Then it is easy to see that $\chi$ is a continuous $k$--algebra 
automorphism.
Let $\varepsilon\co C\to R^c_*$ be the map given by
$\varepsilon(\alpha)=\alpha(e)$ for $\alpha\in C$,
where $e$ is the identity element of $G$.
Then it is also easy to see that $\varepsilon$ 
is a continuous $k$--algebra homomorphism.

\begin{theorem}{\rm (cf Hovey \cite[Section 6.3]{Hovey})}\qua
\label{theorem:Hopf_algebroid_structure}
The pair $(R^c_*,C)$ with 
$(\eta_R,\eta_L,\psi,\chi,\varepsilon)$
is a graded complete Hopf algebroid over $k$.
\end{theorem}


\begin{remark}\rm
Let $C(G,k)$ be the ring of all continuous maps from
$G$ to $k$.
There is an isomorphism 
$C\cong C(G,k)\widehat{\otimes}_k R^c_*$
of complete $k$--algebras
by \fullref{lemma:complete_arguments_change},
and 
$C(G,k)$ is a Hopf algebra over $k$
by \fullref{theorem:Hopf_algebroid_structure}.
The right action of $G$ on $R^c_*$ gives $R^c_*$
a graded (left) $C(G,k)$--comodule algebra structure.
Let $\rho\co R^c_*\to C(G,k)\widehat{\otimes}_k R^c_*$
be the comodule algebra structure map.
In this situation we can construct a split Hopf algebroid 
$(R^c_*, C(G,k)\widehat{\otimes}_k R^c_*)$.
In fact, 
$\rho=\eta_L$ under the above isomorphism,
and the graded complete Hopf algebroid 
$(R^c_*, C)$ is isomorphic to 
$(R^c_*, C(G,k)\widehat{\otimes}_k R^c_*)$.  
\end{remark}

\subsection{Twisted modules}\label{subsection:twisted_modules}

In this subsection we show that there is an equivalence of 
symmetric monoidal categories between the category of
complete $C$--comodules and the category of 
complete twisted $R^c_*$--$G$--modules.  

\begin{definition}\rm
A complete Hausdorff filtered $R^c_*$--module $M$ is said to be
a complete twisted (right) $R^c_*$--$G$--module 
if $G$ acts on $M$ continuously (from the right) such that
$(am)g = a^g\cdot (m)g$ for all $m\in M, a\in R_*, g\in G$. 
\end{definition}

\begin{remark}
The category of complete twisted $R^c_*$--$G$--modules is a symmetric
mon\-oid\-al category under complete tensor product 
$\subrel{R^c_*}{\wtens}$ and unit object $R^c_*$.
\end{remark}

\begin{definition}\rm
A complete Hausdorff filtered $R^c_*$--module $M$ is said to be
a complete (left) $C$--comodule 
if there is a continuous left $R^c_*$--module homomorphism
$\rho_M\co M\to C\subrel{R^c_*}\wtens M$,
which makes co-associativity and co-unity diagrams
commute. 
\end{definition}

\begin{remark}
The category of complete $C$--comodules is a symmetric
monoidal category under complete tensor product  
$\subrel{R^c_*}{\wtens}$ and unit object $R^c_*$.
\end{remark}

\begin{lemma}\label{lemma:functor_c-comod2twisted_mod}
For a complete (left) $C$--comodule $M$,
there is a natural complete twisted (right) $R^c_*$--$G$--module
structure on $M$.
\end{lemma}

\begin{proof}
We denote by $\mbox{\rm ev}(g)\co C\to R^c_*$
the evaluation map at $g\in G$.
Then the map
\[ M\longrightarrow C\subrel{R^c_*}{\wtens} M
   \stackrel{\mbox{\rm ev}(g)\otimes 1}
   {\hbox to 15mm{\rightarrowfill}}
   R^c_*\subrel{R^c_*}{\wtens} M\cong M \]
 defines a twisted $R_*$--$G$--module structure on $M$.
Hence it is sufficient to show that 
the action map $M\times G\to M$ is continuous. 

Let $N$ be an open $R$--submodule of $M$.
Then ${\mathfrak m}^i M\subset N$ for some $i$,
since $M$ is an $R^c_*$--module.
In this case, 
$C(G,R/{\mathfrak m}^i)\otimes_R M/N\cong C(G,M/N)$. 
Then there is an open $R$--submodule $N'$,
which makes the following diagram commute:
\[\xymatrix{
    M \ar[r]\ar[d] & C\wtens M \ar[d]\\
    M/N' \ar[r] & C(G,M/N).
}
\]
We note that for any element of $M/N'$ the image under the bottom
arrow factors through $F(G/U, M/N)$ for some
open normal subgroup $U$ of $G$.
The above commutative diagram gives us the following commutative diagram:
\[\xymatrix{
     M \times G \ar[r]\ar[d] & M \ar[d]\\
     M/N' \times G \ar[r] & M/N. 
   }
\]
By the above remark,
the bottom arrow is continuous.
Hence top arrow is also continuous.
This completes the proof.
\end{proof}

\begin{lemma}\label{lemma:change_groupaction2comodstr}
Let $M$ be an $R^c_*$--module
and $G$ a profinite group.
Then there is an isomorphism of left $R^c_*$--modules
\[ C\subrel{R^c_*}{\wtens}M\cong C(G,M),\]
where the left $R^c_*$--module structure on $C\wtens M$
comes from $\eta_L$ of $C$, 
and the left $R^c_*$--module structure on $C(G,M)$ is given by
$(r\cdot f)(g)=r^g\cdot f(g)$ for $r\in R$, $f\in C(G,M)$, $g\in G$.
\end{lemma}

\begin{proof}
There is an isomorphism 
$C\cong C(G,k)\wtens_k R^c_*$
by \fullref{lemma:complete_arguments_change}.
Hence $C\wtens_{R^c_*}M$ is isomorphic to $C(G,k)\wtens_k M$.
Then we have
\[ \begin{array}{rcl}
    C(G,k)\wtens_k M & \cong & \inverselimit{N}\directlimit{U}
                          F(G/U,k)\otimes_k M/N\\
                & \cong & \inverselimit{N}\directlimit{U}
                          F(G/U,M/N)\\
                & \cong & \inverselimit{N} C(G, M/N)\\
                & \cong & C(G,M).\\
   \end{array}\] 
It is easy to check that this isomorphism respects
the left $R^c_*$--module structures.
\end{proof}

\begin{lemma}\label{lemma:functor_twisted_mod2c-comod}
For a complete twisted (right) $R^c_*$--$G$--module $M$,
there is a natural complete (left) $C$--comodule structure
on $M$.
\end{lemma}

\begin{proof}
The $G$--module structure map $M\times G\to M$
gives a map $M\to C(G,M)$.
By \fullref{lemma:change_groupaction2comodstr},
we obtain a map $M\to C\wtens M$.
If this map is continuous,
it is easy to check that it defines 
a complete $C$--comodule structure on $M$.  
Hence it is sufficient to show that $M\to C(G,M)$
is continuous.
For any open submodule $N$ of $M$ and $g\in G$,
there are open submodule $N_g$ of $M$
and an open neighborhood $U_g$ of $g$
such that $N_g\cdot U_g\subset N$.
Since $G$ is compact,
$G=U_{g_1}\cup\cdots\cup U_{g_n}$.
Let $N'$ be an open submodule such that
$N'\subset N_{g_1}\cap\cdots\cap N_{g_n}$.
Then $N'\cdot G\subset N$.
Hence the map $M\times G\to M\to M/N$
factors through $M/N'\times G$.
This implies that the map 
$M\to C(G,M)\to C(G,M/N)$ factors through $M/N'$.
Hence $M\to C(G,M)$ is continuous.
This completes the proof.  
\end{proof}

\begin{theorem}\label{theorem:eq_cat_twist_comod_over_C}
There is an equivalence of symmetric monoidal categories between
the category of complete twisted (right) $R^c_*$--$G$--modules and
the category of complete (left) $C$--comodules.
\end{theorem}

\begin{proof}
By \fullref{lemma:functor_c-comod2twisted_mod} 
and \fullref{lemma:functor_twisted_mod2c-comod},
there is an equivalence of categories between the category
of complete $C$--comodules and the category of 
complete twisted $R^c_*$--$G$--modules.
It is easy to check that this equivalence respects  
the symmetric monoidal structures. 
\end{proof}

\subsection{Remark on twisted modules}
\label{subsection:Lemma on twisted modules}

In this subsection we let $G=\gal{\fp{n}/\fp{}}\ltimes S_n$. 
Usually, $G$ is called the $n$--th extended Morava stabilizer group,
and it is important 
to study the category of complete twisted $\fp{n}$--$G$--modules.
In this subsection we compare the category of 
complete twisted $\fp{n}$--$G$--modules
and the category of complete twisted $\F$--$G_n$--modules.

There is an exact sequence:
\[   1\to \gal{\F/\fp{n}}\longrightarrow G_n\longrightarrow G\to 1.\] 
Hence, for a twisted $\F$--$G_n$--module $M$,
the submodule 
$\M$ invariant over $\gal{\F/\fp{n}}$
is a twisted $\fp{n}$--$G$--module.
Conversely, for a twisted $\fp{n}$--$G$--module $N$,
we can give $\F\otimes_{\fp{n}}N$ an obvious twisted $\F$--$G_n$--module
structure.

\begin{lemma}\label{lemma:descent_twisted_gal_mod}
For a finite dimensional twisted $\F$--$G_n$--module $M$,
$\M$ is a finite dimensional 
twisted $\fp{n}$--$G$--module,
and $\F\otimes_{\fp{n}}\M$ is isomorphic
to $M$ as a twisted $\F$--$G_n$--module.  
\end{lemma}

\begin{proof}
Let $m=\mbox{\rm dim}_{\F}M$.
We obtain an isomorphism 
$\smash{M\cong \F\otimes_{\fp{n}}\M}$
as twisted $\F$--$\gal{\F/\fp{n}}$--modules,
since we have
$H^1(\gal{\F/\fp{n}};\mbox{\rm GL}_m(\F)) = \{1\}$ 
\mbox{(cf Serre \cite[Proposition~X.1.3]{Serre})}. 
By the above exact sequence,
$\M$ is a twisted $\fp{n}$--$G$--module,
and we see that this is an isomorphism of twisted
$\F$--$G_n$--modules. 
\end{proof}

\begin{remark}\label{remark:continuity_twisted_gal_mod}
Since the cardinality of $M$ is finite,
the action of $G_n$ on $M$ and the action of $G$ on 
$\M$ are continuous by 
\fullref{lemma:topology_structure_G_n}.
\end{remark}

Let $M$ be a profinite twisted $\F$--$G_n$--module.
By \fullref{cor:fund_neigh_twisted_mod},
we can take a fundamental neighborhood system $\{F^{\lambda}M\}$ at $0$
consisting of open $\F$--$G_n$--submodules.
Then $$\M\cong \ \subrel{\longleftarrow}{\lim}
(M/F^{\lambda}M)^{\mbox{\scriptsize $\gal{\F/\fp{n}}$}},$$
and $\M$ is a profinite twisted $\fp{n}$--$G$--module
with filtration $F^{\lambda}(M^{\mbox{\scriptsize $\gal{\F/\fp{}}$}})$,
where $F^{\lambda}(M^{\mbox{\scriptsize $\gal{\F/\fp{n}}$}})$ 
is the kernel of the map $\M\to 
(M/F^{\lambda}M)^{\mbox{\scriptsize $\gal{\F/\fp{n}}$}}$.
Conversely,
for a profinite twisted $\fp{n}$--$G$--module $N$,
we can give $\F\wtens_{\fp{n}}N$ an obvious 
profinite twisted $\F$--$G_n$--module structure.
By \fullref{lemma:descent_twisted_gal_mod}
and \fullref{remark:continuity_twisted_gal_mod},
we obtain the following proposition.

\begin{proposition}\label{prop:eq_cats_twisted_modules}
The functor $M\mapsto \M$
gives an equivalence of symmetric monoidal categories
between the category of profinite twisted $\F$--$G$--modules
and the category of profinite twisted $\fp{n}$--$G$--modules.
The quasi-inverse of this functor is given by
$N\mapsto \F\wtens_{\fp{n}}N$.
\end{proposition}

By \fullref{prop:eq_cats_twisted_modules},
there is no essential difference between
profinite twisted $\F$--$G_n$--modules
and profinite twisted $\fp{n}$--$G$--modules.

\subsection{Reformulation}\label{subsection:reformulation}

In this section we reformulate the results in \cite{Torii-9}.
Set 
\[ \begin{array}{rcl}
    C_{E_*} & = & C(G_{n+1}, E_*^c),\\[1mm]
    C_{K_*} & = & C(G_n, K_*).\\
   \end{array}\]
Then $(E_*^c,C_{E_*})$ and $(K_*,C_{K_*})$ are graded complete 
Hopf algebroids over $\fp{}$ by 
\fullref{theorem:Hopf_algebroid_structure}.

Let $M$ be a profinite $C_{E_*}$--precomodule.
Then $M^c$ is a complete twisted $E_*^c$--$G_{n+1}$--module 
by \fullref{lemma:functor_c-comod2twisted_mod}. 
Note that $M^c=M$ as an abstract $E_*$--module.
Since the $G_{n+1}$ action on $L$ is compatible with the 
$G_{n+1}$--action on $M$,  
$G_{n+1}$ acts on $M\wtens_{E_*} L_*$,
where $L_*=L[u^{\pm 1}]$ is regarded as a discrete module.
We define ${\mathcal F}(M)$ to be the $S_{n+1}$--invariant submodule
of $L_*\wtens_{E_*} M$:
\[ {\mathcal F}(M)= H^0(S_{n+1}; L_*\subrel{~E_*}{\wtens}M). \]
We regard $K_*=\F[w^{\pm 1}]$ as a subring of $L_*=L[u^{\pm 1}]$
by $w=\Phi_0^{-1} u$.
The following lemma was proved in \cite[Lemma~4.2]{Torii-9}.

\begin{lemma}\label{lemma:first_step_lemma}
$H^0(S_{n+1}; L_*)=K_*$.
\end{lemma}

Note that $E_*$ with discrete topology is a $C_{E_*}$--precomodule.

\begin{corollary}
${\mathcal F}(E_*)=K_*$.
\end{corollary}

\begin{lemma}\label{lemma:finite_dimensionality_lemma}
Let $M_{L}$ be a finite dimensional twisted $L_*$--$G_{n+1}$--module.
Then the dimension of $H^0(S_{n+1};M_{L})$ over $K_*$
is finite.  
\end{lemma}

\begin{proof}
We prove the lemma by induction on the dimension of $M_L$.
Suppose that $\mbox{\rm dim}\, M_L=1$. 
If $H^0(S_{n+1}; M_L)=0$, then it is okay.
Suppose that $H^0(S_{n+1}; M_L)\neq 0$.
Take a nonzero $a\in H^0(S_{n+1}; M_L)$.
Then $M_L$ is isomorphic to $L_*$ as a twisted $L_*$--$G_{n+1}$--module.
Hence this case follows from \fullref{lemma:first_step_lemma}.

Suppose that $\mbox{\rm dim}\, M_L=n >1$,
and that the lemma is true for $\smash{M'_L}$ of dimension $<n$.
If $H^0(S_{n+1}; M_L)=0$, then it is okay.
Suppose that $H^0(S_{n+1}; M_L)\neq 0$.
Let $a\in H^0(S_{n+1}; M_L)$ be a nonzero element,
and $N_L$ the $L_*$--submodule generated by $a$.
There is an exact sequence of $K_*$--modules:
\[ 0\to H^0(S_{n+1}; N_L)\longrightarrow 
           H^0(S_{n+1}; M_L)\longrightarrow
           H^0(S_{n+1}; M_L/N_L).\]
By hypothesis of the induction,
the dimension of $H^0(S_{n+1}; M_L/N_L)$ is finite and
$\mbox{\rm dim}\ H^0(S_{n+1}; N_L)=1$.
Hence we obtain that
$\mbox{\rm dim}\, H^0(S_{n+1}; M_L)$ 
is finite.
This completes the proof. 
\end{proof}

\begin{remark}\label{remark:dimension_estimate}\rm
More precisely,
we see that 
$\mbox{\rm dim}_{K_*}H^0(S_{n+1};N_L)\le \mbox{\rm dim}_{L_*} N_L$
by the proof of \fullref{lemma:finite_dimensionality_lemma}.
\end{remark}

\begin{corollary}\label{cor:finite_dimensionality_lemma}
If $M$ is a finitely generated discrete $C_{E_*}$--precomodule,
then the dimension of ${\mathcal F}(M)$ over $K_*$
is finite. 
\end{corollary}

The following lemma is fundamental
on the topology of $G_n$.   

\begin{lemma}{\rm (cf Hovey \cite[Theorem~A.2]{Hovey})}\qua
\label{lemma:topology_structure_G_n}
A subgroup of $G_n$ is open if and only if its index in $G_n$ 
is finite.
\end{lemma}

Let $h\in G_n$ and $\sigma$ the image of the projection
$G_n\to \Gal$.
For $g\in S_{n+1}$,
$g\sigma = \sigma g^{\sigma}$ in $G_{n+1}$,
and $g h = h g ^{\sigma}$ in ${\mathcal G}$.
Hence the following diagram commutes for all $g\in S_{n+1}$:
\[\xymatrix@C=40pt{
    L_*\subrel{E_*}{\wtens}M \ar[r]^{\sigma\otimes h} \ar[d]_-{g\otimes g} & L_*\subrel{E_*}{\wtens}M \ar[d]^-{g^{\sigma}\otimes g^{\sigma}}\\
    L_*\subrel{E_*}{\wtens}M \ar[r]^{\sigma\otimes h} & L_*\subrel{E_*}{\wtens}M.
  }
\]
This diagram induces an action of $G_n$ on ${\mathcal F}(M)$,
and it is easy to check that ${\mathcal F}(M)$ is 
a twisted $K_*$--$G_n$--module.

\begin{lemma}\label{lemma:twisted_module_str_F}
If $M$ is a finitely generated discrete $C_{E_*}$--precomodule,
then ${\mathcal F}(M)$ has a natural complete twisted 
$K_*$--$G_n$--module structure. 
\end{lemma}

\begin{proof}
By \fullref{cor:finite_dimensionality_lemma},
${\mathcal F}(M)$ is a twisted $K_*$--$G_n$--module 
of finite dimension.
Then the action of $G_n$ is continuous by 
\fullref{lemma:topology_structure_G_n}.
\end{proof}

If $M$ is a complete $\smash{C_{E_*}}$--precomodule,
then there is a fundamental system $\{F^{\lambda}M\}$
of (open) neighborhoods at $0$ consisting of $E_*$--$G_{n+1}$--submodules
by \fullref{cor:fund_neigh_twisted_mod}.
Hence there is an isomorphism
\[ {\mathcal F}(M)\cong\
   \inverselimit{\lambda} {\mathcal F}(M/F^{\lambda}M).\]
We give ${\mathcal F}(M)$ the inverse limit topology.
Note that this topology is independent of a choice
of fundamental system of neighborhood at $0$.
Furthermore, if $M$ is profinite,
then ${\mathcal F}(M)$ is also profinite by 
\fullref{cor:finite_dimensionality_lemma},
and complete twisted $K_*$--$G_n$--module by 
\fullref{lemma:twisted_module_str_F}.
Hence we obtain the following proposition.

\begin{proposition}\label{prop:proftoprof}
${\mathcal F}$ defines a symmetric monoidal functor 
from the category of profinite $C_{E_*}$--precomodules
to the category of profinite $C_{K_*}$--comodules.
\end{proposition}

\begin{proof}
Since the construction of twisted $K_*$--$G_n$--module
structure on ${\mathcal F}(M)$ is natural,
we see that ${\mathcal F}$ defines a functor from 
the category of profinite $C_{E_*}$--precomodules
to the category of profinite twisted $K_*$--$G_n$--modules,
which is equivalent to the category of profinite $C_{K_*}$--comodules
by \fullref{theorem:eq_cat_twist_comod_over_C}.
It is easy to check that the functor ${\mathcal F}$
respects the monoidal structures.
\end{proof}

\begin{definition}\rm
Let $\smash{\mathcal{C}_{\mathbb{E}}^f}$ (resp.\ $\smash{\mathcal{C}_{\mathbb{K}}^f}$)
be the category of finitely generated discrete
$\smash{C_{E_*}}$--precomodules (resp.\ $\smash{C_{K_*}}$-(pre)comodules).
Define $\mathcal{C}_{\mathbb{E}}$ 
(resp.\ $\mathcal{C}_{\mathbb{K}})$ to be the procategory
of $\smash{\mathcal{C}_{\mathbb{E}}^f}$ (resp.\ $\smash{\mathcal{C}_{\mathbb{K}}^f}$),
that is, the category of (small) cofiltered system 
of objects in $\smash{\mathcal{C}_{\mathbb{E}}^f}$ 
(resp.\ $\smash{\mathcal{C}_{\mathbb{K}}^f}$).
These are symmetric monoidal categories.
\end{definition}

For a finite $Z$, 
$E^*(Z)$ is a natural finitely generated discrete 
$C_{E_*}$--precomodule 
by \fullref{prop:precomdule_cohomology}.
Furthermore, if $Z$ is a finite CW-complex,
then $E^*(Z)$ is a $C_{E_*}$--precomodule algebra.

\begin{definition}\rm
We define $\mathbb{E}^*(X)\in \mathcal{C}_{\mathbb{E}}$
to be the system 
\[ \{E^*(Z)\}_{Z\in\Lambda(X)} \]
indexed by $\Lambda(X)$. 
We also define $\mathbb{K}^*(X)\in \mathcal{C}_{\mathbb{K}}$
by the same manner.
\end{definition}

\[\eqaligntop{
    \inverselimit{} \mathbb{E}^*(X)&\cong  E^*(X),\tag*{\hbox{Then we have}}\cr
    \inverselimit{} \mathbb{K}^*(X)&\cong  K^*(X),
}\]
as profinite $C_{\mathbb{E}_*}$--precomodules
and profinite $C_{\mathbb{K}_*}$-(pre)comodules, respectively.
By \fullref{cor:finite_dimensionality_lemma}
and \fullref{lemma:twisted_module_str_F},
we can extend the functor $\mathcal{F}$
from $\mathcal{C}_{\mathbb{E}}$ to $\mathcal{C}_{\mathbb{K}}$
by obvious way:
\[ \mathcal{F}\co \mathcal{C}_{\mathbb{E}}\longrightarrow
                \mathcal{C}_{\mathbb{K}}.\]
Note that $\mathcal{F}$ is a monoidal functor.

By \cite[Theorem~4.1]{Torii-9},
the generalized Chern character \eqref{eq:generalized_Chern_character} 
\[ \Theta\co E^*(X)\longrightarrow L^*(X) \]
induces a natural isomorphism
of twisted $L_*$--${\mathcal G}$--modules:
\[  L_* \otimes_{E_*}E^*(Z)\stackrel{\cong}{\longrightarrow}
    L_*\otimes_{K_*}K^*(Z)\]
for finite $Z$.
The following theorem is a reformulation of
\cite[Corollary~4.3]{Torii-9}.

\begin{theorem}\label{thm:reformulation_stab}
For any spectrum $X$,
the generalized Chern character $\Theta$
induces a natural isomorphism in $\mathcal{C}_{\mathbb{K}}$:
\[ {\mathcal F}(\mathbb{E}^*(X))\cong \mathbb{K}^*(X). \]
If $X$ is a space, then this is an isomorphism 
of cofiltered systems of finite $C_{K_*}$--comodule algebras.
\end{theorem}

\section{Milnor operations}
\label{section:Milnor_operations}

\subsection{Complete co-operation ring}
\label{subsection:Co-operation ring}

In this section we let $A=E_{n+k}/I_n$ for some $k\ge 0$.
Hence $A=E$ if $k=1$ and $A=K$ if $k=0$.
The coefficient ring $A_*=\F\power{u_n,\ldots,u_{n+k-1}}[u_A], 
\ |u_A|=-2$
is a graded complete Noetherian local ring with 
maximal ideal ${\mathfrak m}_A=(u_n,\ldots,u_{n+k-1})$.
Put $G_A=G_{n+k}$ and $C_{A_*}=C(G_A,A_*)$.
We denote by $A\widehat{\wedge}A$
the $K(n+k)$--localization of $A\wedge A$.
Since $A$ is a commutative ring spectrum,
So is $A\widehat{\wedge}A$.
We define a graded commutative ring $\acha$ to be
$\pi_*(A\widehat{\wedge}A)$.

Since $A$ is Landweber exact over $\pn$, 
there is an isomorphism of commutative $\fp{}$--algebras:
\begin{equation}\label{eq:isomorphism_landweber_exact_ring}
   \pi_*(A\wedge A)\cong A_*\otimes_{\pn_*}\pn_*(\pn)
   \otimes_{\pn_*}A_*.
\end{equation}

\begin{lemma}\label{lemma:acha_as_rings}
There is an isomorphism of graded commutative $\fp{}$--algebras   
\[ \acha\cong A_*^c\wtens_{\pn_*}
   {\pn_*(\pn)}\wtens_{\pn_*}A_*^c,\]
where $A_*^c$ is a graded topological ring $A_*$ 
with ${\mathfrak m}_A$--adic topology.
\end{lemma}

\begin{proof}
By \cite[Proposition~7.10(e)]{Hovey-Strickland},
we see that $\acha$ is the $I_{n+k}$--adic completion 
of $\pi_*(A\wedge A)$.
Hence the lemma follows from 
the isomorphism~\eqref{eq:isomorphism_landweber_exact_ring}.
\end{proof}

By \fullref{lemma:acha_as_rings},
we see that $\acha$ has a graded complete  Hopf algebroid structure
induced from $\pn_*(\pn)$.
We say that $\acha$ is the complete co-operation ring of $A$.

Let $\Lambda_{\mathbf Z}$ be the graded commutative
algebra over ${\mathbf Z}$ generated by $a_{(i)}$ for $0\le i< n$,
where the degree of $a_{(i)}$ is $2p^i-1$.
Hence $\Lambda_{\mathbf Z}$ is an exterior algebra. 
For an evenly graded commutative ring $R_*$,
we set 
$\Lambda_{R_*} = R_*\otimes \Lambda_{\mathbf Z}$.
There is an isomorphism of commutative $\fp{}$--algebras
\[ \pn_*(\pn)\cong \pn_*[t_1,t_2,\ldots]\otimes\Lambda_{\mathbf Z},\]
where $|t_i|=2(p^i-1)$.
Let $C_{\pn_*}$ be the $\pn_*$--subalgebra of $\pn_*(\pn)$
generated by $t_1,t_2,\ldots$:
\[ C_{\pn_*}=\pn_*[t_1,t_2,\ldots] .\]
Then it is known that $C_{\pn_*}$ is a sub-Hopf algebroid of 
$\pn_*(\pn)$ \mbox{(cf Wurgler \cite{Wurgler})}.
Hence we can give $A^c_*\wtens_{\pn_*}C_{\pn_*}\wtens_{\pn_*}A^c_*$ 
the induced graded complete Hopf algebroid structure. 

\begin{lemma}
There is an isomorphism of graded complete Hopf algebroids over $\fp{}$:
\[ (A^c_*,C_{A_*})\cong 
   (A^c_*,A^c_*\wtens_{\pn_*}C_{\pn_*}\wtens_{\pn_*}A^c_*).\]
\end{lemma}

\begin{proof}
We let $D_*  
=\pi_*(E_{n+k}\widehat{\wedge}E_{n+k})$,
where $E_{n+k}\widehat{\wedge}E_{n+k}$ is 
the $K(n+k)$--localization of $E_{n+k}\wedge E_{n+k}$.
Then 
$D_* 
\cong E_{n+k,*}^c\wtens_{BP_*}BP_*(BP)\wtens_{BP_*}E_{n+k,*}^c$.
Hence $(E_{n+k,*}^c,D_*)$ is 
a graded complete Hopf algebroid over ${\mathbf Z}_{(p)}$.
Furthermore, 
$$(E_{n+k,*}^c, D_*) 
\cong (E_{n+k,*}^c, C(G_{n+k},E_{n+k,*}^c))$$ \cite{Devinatz,Hovey}.
Hence $$D_*/I_n\cong A^c_*\wtens_{\pn_*}C_{\pn_*}\wtens_{\pn_*}A^c_*$$
by \fullref{lemma:acha_as_rings}.
By \fullref{lemma:complete_arguments_change},
$C(G_{n+k},E_{n+k,*})\cong C(G_{n+k},{\mathbf Z})\wtens E_{n+k,*}$.
This implies that
$C(G_{n+k},E_{n+k,*})/I_n\cong C(G_{n+k},A_*)$.
Hence we have the isomorphism
$C_{A_*}\cong A^c_*\wtens_{\pn_*}C_{\pn_*}\wtens_{\pn_*}A^c_*$.
We can check that this isomorphism induces
the desired isomorphism of Hopf algebroids. 
\end{proof}

\begin{corollary}\label{cor:structureonacha}
There is an isomorphism of graded complete commutative $\fp{}$--algebras:   
\[ \acha \cong 
   C_{A_*}\wtens\Lambda_{\mathbf Z}.\]
\end{corollary}

Recall that $A^*(X)^c=A^c_*\wtens_{A_*}A^*(X)$.
The natural $\pn_*(\pn)$--comodule structure on $\pn^*(Z)$
gives a natural $\acha$--comodule structure on $A^*(Z)^c$, for any finite spectrum $Z$.
By \fullref{lemma:complete_arguments_change} and 
\fullref{cor:structureonacha}, 
this induces an $\acha$--comodule structure on $A^*(X)^c$:
\[ \rho\co A^*(X)^c\longrightarrow \acha\wtens_{A_*^c} A^*(X)^c.\]  
If $X$ is a space, 
then $\rho$ defines an $\acha$--comodule algebra structure
on $A^*(X)^c$.

In the following of this subsection we describe 
the comultiplication $\psi $ on $a^A_{(i)}$.  
For $0\le i<n$, we set
\[ 
      b^A_{(i)}   =   u_A^{p^i} \otimes a_{(i)}.
\]
Then $|b^A_{(i)}|=-1$ for all $i$.
In particular, $u_A=u$ if $A=E$ and $u_A=w$ if $A=K$.
We put 
\begin{align*}t_A(X)&=\sum_{i\ge 0}{}^{\eta_{L*}F_{A}}~~t^A_i X^{p^i} \in C_{A_*}\power{X}&\\
b_A(X)&=\sum_{i=0}^{n-1}b^A_{(i)}X^{p^i} \qquad\ \, \in \acha[X],&\tag*{\hbox{and}}
\end{align*} 
where $F_A$ is the base change of the universal deformation
$F_{n+k}$ on $E_{n+k,0}$ to $E_{n+k,0}/I_n$.

The comodule algebra structure map
$\rho \co A^*(\mathcal{X})^c\to\acha\wtens_{A_*^c}A^*(\mathcal{X})^c$ 
on $x_A$ 
is given by the following lemma (cf \cite[Section 14]{Boardman}).

\begin{lemma}\label{lemma:action_xA}
$\psi_A(x_A)=t_A(x_A)$.  
\end{lemma}


%

The comodule algebra structure map
$\rho \co A^*(\mathcal{Y})^c\to\acha\wtens_{A_*^c}A^*(\mathcal{Y})^c$ 
on $y_A$ 
is given by the following lemma (cf \cite[Section 14]{Boardman}).

\begin{lemma}\label{lemma:action_on_yA}
$\rho (y_A)= 1\otimes y_A + b_A(x_A) $.
\end{lemma}


Let $i_l$ and $i_r$ be the left and right inclusion
of $\acha$ into $\acha\wtens_{A^c_*}\acha$.
The comultiplication map $\psi$ on $b^A_{(i)}$ is encoded in
the following lemma.

\begin{lemma}\label{lemma:encoded_psi_bi}
$\psi(b_A(X))\equiv i_r(b_A)(X)+ i_l(b_A)(i_r(t_A)(X)) \mod (X^{p^n})$.
\end{lemma}

\begin{proof}
This follows from the fact that
$(\psi\otimes 1)\rho(y_A)=(1\otimes\rho)\rho(y_A)$
and \fullref{lemma:action_xA} and \fullref{lemma:action_on_yA}.
\end{proof}

\begin{lemma}\label{lemma:mod_x^p^n_calculation}
Let $F$ be a $p$--typical formal group law of strict height at least $n$
over an $\fp{}$--algebra $R$.
Then for $a_i\in R\ (0\le i<n)$,
\[ \sum_{i=0}^{n-1}{}^F~a_i X^{p^i}\equiv 
  a_0 X+a_1X^p+\cdots +a_{n-1}X^{p^{n-1}} \mod (X^{p^n}).\]
\end{lemma}

\begin{proof}
This follows from the fact that 
$F(X,Y)\equiv X+Y$ mod $(X,Y)^{p^n}$.
\end{proof}

The following theorem describes the structure of  
graded complete Hopf algebroid $(A^c_*,\acha)$.

\begin{theorem}\label{theorem:structure_acha}
The pair $(A^c_*,\acha)$ is a graded complete Hopf algebroid over $\fp{}$.
There is an extension of graded complete Hopf algebroids
\[ C_{A_*}\longrightarrow \acha \longrightarrow \Lambda_{A^c_*},\]  
where the algebra $\smash{\Lambda_{A^c_*}=A_*^c\otimes \Lambda(b^A_{(0)},\ldots,b^A_{(n-1)})}$ 
is an exterior Hopf algebra over $A_*^c$ generated by primitive elements
$\smash{b^A_{(i)}}$ for $0\le i<n$.
The comultiplication $\psi$ and the counit $\varepsilon$
on $\smash{b^A_{(i)}}$ for $0\le i<n$ are given as follows:
\[  \begin{array}{rcl}
       \psi(b^A_{(i)}) & = & 1\otimes b^A_{(i)} 
     + {\displaystyle\sum_{j=0}^{i}}~~b^A_{(j)}\otimes 
       (t^A_{i-j})^{p^j},\\
      \varepsilon(b^A_{(i)}) &= & 0.\\
    \end{array}\]
\end{theorem}

\begin{proof}
The comultiplication $\psi$ on $b^A_{(i)}$
is obtained by \fullref{lemma:action_on_yA} and 
\fullref{lemma:mod_x^p^n_calculation}. 
\end{proof}


%
\subsection[Exterior algebras]{Exterior algebras $\Lambda_{E_*}$ and $\Lambda_{K_*}$}
\label{subsection:Lambda_E_Lambda_K}

%
%

We can give $\Lambda_{A^c_*}$ a structure of right 
$C_{A_*}$--comodule algebra by
\[ \rho_{C,\Lambda}^{\rm op}\co \Lambda_{A^c_*}
   \stackrel{i_{\Lambda}}{\longrightarrow}
   \acha\stackrel{\psi}{\longrightarrow}
   \acha\wtens_{A^c_*}\acha
   \stackrel{\pi_{\Lambda}\otimes\pi_C}
   {\hbox to 10mm{\rightarrowfill}}
   \Lambda_{A^c_*}\wtens_{A^c_*} C_{A_*},\]
where $i_{\Lambda}$ is the canonical inclusion,
$\pi_C=1_C\otimes \varepsilon_{\Lambda}$ and
$\pi_{\Lambda}=\varepsilon_C\otimes 1_{\Lambda}$.
Hence $\Lambda_{A_*}$ is a profinite right $C_{A_*}$--precomodule algebra.

\begin{lemma}\label{lemma:coaction_lambda_c}
$\rho_{\Lambda,C}^{\rm op}(b_A(X))\equiv b_A(t_A(X))\mod (X^{p^n})$.
\end{lemma}

\begin{proof}
This follows from \fullref{lemma:encoded_psi_bi}.
\end{proof}

The left $A^c_*$--module homomorphism
$\mbox{\rm ev}(g)\circ\chi\co C_{A_*}\to C_{A_*}\to A^c_*$
defines a right action of $G_A$ on $\Lambda_{A^c_*}$
by 
\[ \Lambda_{A^c_*}\stackrel{\rho_{\Lambda,C}^{\rm op}}
   {\hbox to 10mm{\rightarrowfill}}
   \Lambda_{A^c_*}\wtens_{A^c_*} C_{A_*}\stackrel
   {1\otimes(\mbox{\scriptsize\rm ev}(g)\circ\chi)}
   {\hbox to 20mm{\rightarrowfill}} \Lambda_{A^c_*} .\]
Then $\Lambda_{A^c_*}$ is a twisted $A^c_*$--$G_A$--module.


\begin{corollary}\label{cor:action_g_b} 
For $g\in G_A$,
$b_A{}^g(X) 
\equiv b_A(t_A(g)^{-1}(X))$ mod $(X^{p^n})$.
\end{corollary}  

\begin{proof}
This follows from \fullref{lemma:coaction_lambda_c}
and $t_A(g^{-1})^g(X)=t_A(g)^{-1}(X)$.
\end{proof}


Since $\Lambda_{E_*}$ is a twisted $E_*$--$G_{n+1}$--module,
$\Lambda_{L_*}=L_*\otimes_{E_*}\Lambda_{E_*}$
is a twisted $L_*$--${\mathcal G}$--module.
We define $\smash{\widehat{b}(X)=\sum_{i=0}^{n-1}\widehat{b}_{(i)}X^{p^i}
\in \Lambda_{L_*}[X]}$ by
\[ \widehat{b}(X)\equiv b_E(\Phi^{-1}(X)) \mod (X^{p^n}).\]

\begin{lemma}\label{lemma:invariance_widehat_b}
For any $g\in G_{n+1}$,
$\widehat{b}\,{}^g(X)=\widehat{b}(X)$.  
\end{lemma}

\begin{proof}
By definition and \fullref{cor:action_g_b}, 
$\widehat{b}\,{}^g(X) \equiv  
b_E\circ t(g)^{-1}\circ (\Phi^{-1})^g(X)$
mod $(X^{p^n})$.
By the diagram~\eqref{eq:commutative_diagram_g_n+1_degeneration} 
in \fullref{thm:degeneration_main},
$\Phi^g\circ t(g)(X)= \Phi(X)$.
This implies that $t(g)^{-1}\circ (\Phi^g)^{-1}(X)=\Phi^{-1}(X)$.
Hence $\widehat{b}\,{}^g(X)\equiv b\circ \Phi^{-1}(X)$ mod $(X^{p^n})$.
\end{proof}

By \fullref{lemma:invariance_widehat_b},
we see that the coefficients of $\widehat{b}(X)$
are invariant under the action of $G_{n+1}$.

\begin{lemma}\label{lemma:iso_f_lambda2lambda}
${\mathcal F}(\Lambda_{E_*})= 
K_*\otimes\Lambda(\widehat{b}_{(0)},\ldots,\widehat{b}_{(n-1)})$
as a graded commutative ring.
\end{lemma}

\begin{proof}
We have $\smash{\widehat{b}_{(0)},\ldots,\widehat{b}_{n-1}
\in {\mathcal F}(\Lambda_{E_*})}$.
Since $\widehat{b}_{(i)}$ is a linear combination of
$\smash{b_{(i)}^E,\ldots,b^E_{(n-1)}}$,
we see that 
$\smash{K_*\otimes\Lambda(\widehat{b}_{(0)},\ldots,\widehat{b}_{(n-1)})  
\subset {\mathcal F}(\Lambda_{E_*})}$.
Then the lemma follows from the fact that
$\mbox{\rm dim}_{K_*}{\mathcal F}(\Lambda_{E_*})\le 2^n$
by \fullref{remark:dimension_estimate}.
\end{proof}


Recall that 
$t_K(h)(X)$ is the automorphism $t_K(h)\co H_n\longrightarrow H_n^h=H_n$ corresponding to $h\in G_n$.

\begin{lemma}\label{lemma:action_G_n_lambda}
For any $h\in G_n$, we have 
$\widehat{b}\,{}^h(X) = \widehat{b}\circ t_K(h)^{-1}(X)$. 
\end{lemma}   

\begin{proof}
By definition and the fact that
$G_n$ acts on $L$ as Galois group, we have
$\widehat{b}\,{}^h(X)\equiv b\circ(\Phi^{-1})^h(X)$ mod $(X^{p^n})$.
By the diagram~\eqref{eq:commutative_diagram_g_n_degeneration}
in \fullref{thm:degeneration_main},
$t_K(h)\circ\Phi(X) =\Phi^h(X)$. 
This implies that
$(\Phi^h)^{-1}(X)=\Phi^{-1}\circ t_K(h)^{-1}(X)$.
Hence the congruence $\widehat{b}\,{}^h(X)\equiv b\circ 
\Phi^{-1}\circ t_K(h)^{-1}(X)$ mod $(X^{p^n})$ holds.
\end{proof}

\begin{theorem}\label{thm:iso_bet_lambdas}
As a $C_{K_*}$--comodule,
${\mathcal F}(\Lambda_{E_*})$ is isomorphic to $\Lambda_{K_*}$.
\end{theorem}

\begin{proof}
The map $\widehat{b}_{(i)}\mapsto b^K_{(i)}$ 
gives an isomorphism of twisted $K_*$--$G_n$--modules
by \fullref{cor:action_g_b},
\fullref{lemma:iso_f_lambda2lambda} and
\fullref{lemma:action_G_n_lambda}.
\end{proof}

\subsection{Milnor operations}\label{subsection:milnot_op}

Let $A=E_{n+k}/I_n$ for some $k\ge 0$.
In this section we study Milnor operations in $A$.
We abbreviate $C_{A_*}$ to $C$ and
$\Lambda_{A_*^c}$ to $\Lambda$.
In this section we discuss in the category of 
complete Hausdorff filtered $A_*^c$--modules.
We recall that $\Lambda$ is a Hopf algebra 
such that $\smash{b^A_{(i)}}$ is primitive for all $i$.
We take monomials of $\smash{b^A_{(i)}}$ as a basis of $\Lambda$,
and denote the dual of $\smash{b^A_{(i)}}$ by $Q_i^A$
in the dual basis.
Then the monomials of $Q^A_{i}$ form the dual basis. 
We call $Q^A_i$ the Milnor operations.

Let $M$ be a left $\Lambda$--comodule 
with comodule structure map $\rho$.
Then the Milnor operation $Q^A_i$ defines a 
$A_*^c$--module homomorphism as follows:
\[ M\stackrel{\rho}{\longrightarrow} \Lambda\wtens M
    \stackrel{Q^A_{i}\otimes 1_M}{\hbox to 15mm{\rightarrowfill}} M .\]
We abbreviate this homomorphism also to $Q^A_i$.
Note that we write the action of $Q^A_i$ from the right: 
if $\smash{\rho(x)=1\otimes x + \sum_i a^A_{(i)}\otimes x_i +\cdots}$,
then $\smash{(x)Q^A_i=(-1)^{|x|+1}x_i}$. 
There is a relation in the endomorphism ring
of $M$ for any $i$ and $j$:
\begin{equation}\label{eq:exterior_relation_Q}
Q^A_i Q^A_j + Q^A_j Q^A_i = 0.
\end{equation}
In particular, $Q^A_iQ^A_i=0$.
Conversely, if there are $A^c_*$--module homomorphisms 
$Q^A_i$ for $0\le i<n$ such that \eqref{eq:exterior_relation_Q} holds,
then we can construct a $\Lambda$--comodule structure on $M$,
and this construction gives an equivalence of categories.

The category of complete $\Lambda$--comodules
is symmetric monoidal under complete tensor product 
$\wtens_{A^c_*}$ and unit object $A^c_*$.

\begin{lemma}
Let $M$ and $N$ be complete $\Lambda$--comodules.
For any $x\in M$ and $y\in N$,
$(x\otimes y)Q^A_i= x\otimes (y)Q^A_i+ (-1)^{|y|} (x)Q^A_i\otimes y$
in $M\wtens N$.
\end{lemma}

\begin{proof}
Let $\smash{\rho_M(x)=1\otimes x+\sum_i a_{(i)}^A\otimes x_i+\cdots}$
with $x_i=\smash{(-1)^{|x|+1}(x)Q^A_i}$, 
and let $\smash{\rho_N(y)=1\otimes y+\sum_i a_{(i)}^A\otimes y_i+\cdots}$
with $y_i=\smash{(-1)^{|y|+1}(y)Q^A_i}$.
Then
$$\rho_{M\wtens N}(x\otimes y)=
1\otimes x\otimes y+ (-1)^{|x|}\sum_i a^A_{(i)}\otimes x\otimes y_i
+\sum_i a^A_{(i)}\otimes x_i\otimes y+\cdots.$$
Hence we have
$(x\otimes y)Q^A_i=(-1)^{|x|+|y|+1} ((-1)^{|x|}x\otimes y_i+x_i\otimes y)$,
which equals $x\otimes (y)Q^A_i+ (-1)^{|y|}(x)Q^A_i\otimes y$.
\end{proof}

We say that a natural endomorphism $Q$ of complete $A^c_*$--modules
is a derivation of odd degree  with respect to exterior products if  
$(x\otimes y)Q=x\otimes (y)Q+ (-1)^{|y|}(x)Q\otimes y$
for any $x\in M$ and $y\in N$.
Hence the Milnor operations $Q^A_i$ is a derivation of 
odd degree with respect to exterior products.

Let $\Lambda^*$ be the dual module of $\Lambda$:
$\smash{\Lambda^*=\hom{A^c_*}(\Lambda,A^c_*)}$.
Then $\Lambda^*$ is also a Hopf algebra over $A^c_*$,
and $\smash{\Lambda^*\cong A^c_*\otimes\Lambda(Q^A_0,\ldots,Q^A_{n-1})}$
such that $Q^A_i$ are primitive for all $i$.
Recall that 
$\Lambda$ is a twisted $A^c_*$--$G_A$--module.
We can also define a twisted $A^c_*$--$G_A$--module structure on 
$\Lambda^*$ by
\[ (\lambda)(\theta^g) = ((\lambda\cdot g^{-1})\theta)g,\]
for $\theta\in\Lambda^*, g\in G_A, \lambda\in\Lambda$.

\begin{lemma}\label{lemma:action_g_qai}
For $g\in G_A$, 
\[ (Q^A_i)^g= \sum_{j=i}^{n-1} t^A_{j-i}(g)^{p^i} Q^A_j. \]
\end{lemma}

\begin{proof}
This follows from \fullref{cor:action_g_b}.
\end{proof}

Let $M$ be a profinite $\acha$--comodule.
Then $M$ is a twisted $A^c_*$--$G_A$--module and $\Lambda$--comodule.
The following proposition gives us an interaction 
of the actions of $G_A$ and $Q^A_i$ on $M$.

\begin{lemma}
Let $M$ be a profinite $\acha$--comodule.
For $x\in M$ and $g\in G_A$, 
\[ ((x)Q^A_i)g= ((x)g)(Q^A_i)^g .\]
\end{lemma}

\begin{proof}
By \fullref{lemma:action_g_qai},
we see that the map
\[ \theta_1\co \acha\stackrel{\psi}{\longrightarrow}
   \acha\wtens\acha\stackrel{(Q_i^A\circ\pi_{\Lambda})\otimes 1}
                            {\hbox to 15mm{\rightarrowfill}}     
   \acha\stackrel{\mbox{\scriptsize\rm ev}(g)\circ\pi_C}
                            {\hbox to 15mm{\rightarrowfill}} A^c_* \]
is equal to the map
\[ \theta_2\co \acha\stackrel{\psi}{\longrightarrow}
   \acha\wtens\acha\stackrel{(\mbox{\scriptsize\rm ev}(g)\circ\pi_C)\otimes 1}
                            {\hbox to 15mm{\rightarrowfill}} 
   \acha\stackrel{(Q^A_i)^g\circ\pi_{\Lambda}}
                            {\hbox to 15mm{\rightarrowfill}} A^c_*.\] 
Hence $((x)Q^A_i)g=(\theta_1\circ\rho)(x)=
       (\theta_2\circ\rho)(x)=((x)g)(Q^A_i)^g$.
\end{proof}

These are all relations on the $\acha$--comodule $M$
between the $G_A$--action and the $\Lambda^*$--action.
We give interpretation of these relations 
in terms of comodule structures in \fullref{subsection:compatibility}.

In the following lemma we show that 
a derivation of odd degree with respect to exterior products
in the category of stable cohomology operations of $K^*(-)$
is characterized by the action on $y_K\in K^*(\mathcal {Y})$.

\begin{lemma}
Let $Q$ be an odd degree stable cohomology operation 
of $K^*(-)$.
Suppose that $Q$ is a derivation with respect to exterior product.
Then $Q$ is characterized by the action on $y_K\in K^1({\mathcal Y})$. 
\end{lemma}

\begin{proof}
A stable cohomology operation $Q\in K^*(K)$ is a derivation 
if and only if $Q$ is primitive in $K^*(K)$.
Since $\kk$ is free over $K_*$,
the primitive submodule $P(K^*(K))$ 
is the dual of the indecomposable quotient $Q(\kk)$
of the co-operation ring $\kk$.
Recall the isomorphism $\kk\cong C_{K_*}\otimes_{K_*}\Lambda_{K_*}$.
Then we have $Q(\kk)\cong Q(C_{K_*})\oplus Q(\Lambda_{K_*})\cong Q(\Lambda_{K_*})$,
and $Q(\Lambda_{K_*})$ 
is isomorphic to $K_*\{a^K_{(0)},\ldots,a^K_{(n-1)}\}$.
Hence $Q$ is a linear combination 
$\smash{\sum_{i=0}^{n-1} q_i Q^K_i}$ with $q_i\in K_*$.
Since we know that $\smash{Q^K_i(y_K)=x_K{}^{p^i}}$,
we have $\smash{Q(y_K)=\sum_{i=0}^{n-1}q_i x_K{}^{p^{\smash[b]{i}}}}$
in $K^*({\mathcal Y})$.
Since $\smash{x_K{}^{p^i}}$ for $0\le i<n-1$ are linearly independent,
this uniquely determines $q_i$.
Hence $Q$ is characterized by the action of $y_K$.
\end{proof}

\subsection[Complete comodules over the complete co-operation ring of A]{Complete $\acha$--comodules}
\label{subsection:compatibility}

Let $A=E_{n+k}/I_n$ for some $k\ge 0$.
In this section we give a description of 
complete $\acha$--comodules in terms of 
$C_{A_*}$--comodule structure and $\Lambda_{A_*^c}$--comodule structure.
In this section we discuss in the category of 
complete Hausdorff filtered $A_*^c$--modules,
and abbreviate $C_{A_*}$ to $C$ and
$\Lambda_{A_*^c}$ to $\Lambda$.

Let $M$ be a complete $\acha$--comodule with
$\rho_M\co M\to \acha\wtens M$.
By \mbox{\fullref{theorem:structure_acha}},
$\acha\cong C\wtens \Lambda$ as an $\fp{}$--algebra,
and there is an extension of complete Hopf algebroids:
\begin{equation}\label{eq:exact_sequence_rchr}
 C\longrightarrow \acha
      \stackrel{\pi_{\Lambda}}{\longrightarrow} \Lambda.
\end{equation}
Hence $M$ is a $\Lambda$--comodule by
\[ \rho_{\Lambda,M}\stackrel{\rho_M}{\longrightarrow}
   \acha\wtens M\stackrel{\pi_{\Lambda}}{\longrightarrow} 
   \Lambda\wtens M.\]
The counit of $\Lambda$ induces a morphism of Hopf algebroid
$\pi_C\co \acha\to C$, which is a splitting of 
the above extension~\eqref{eq:exact_sequence_rchr}. 
Then $M$ is also a $C$--comodule by
\[ \rho_{C,M}\co M\stackrel{\rho_M}{\longrightarrow}
   \acha\wtens M\stackrel{\pi_C}{\longrightarrow}
   C\wtens M.\]
We recall that $\Lambda$ is a (left) $C$--comodule algebra by
the structure map 
\[ \rho_{C,\Lambda}\co \Lambda\stackrel{i_{\Lambda}}{\longrightarrow}
   C\wtens \Lambda\stackrel{\psi}{\longrightarrow}
   (C\wtens \Lambda)\wtens 
   (C\wtens \Lambda)
   \stackrel{\pi_{\Lambda}\otimes\pi_C}
   {\hbox to 10mm{\rightarrowfill}}
   \Lambda\wtens C
   \stackrel{\tau}{\longrightarrow}C\wtens \Lambda,\]   
where $i_{\Lambda}$ is the canonical inclusion and
$\tau$ is given by $\lambda\otimes c\mapsto \chi(c)\otimes\lambda$.
For a complete $C$--comodule $M$, 
we denote by $\smash{\rho_{C,\Lambda\wtens M}}$
the $C$--comodule structure map of the tensor product 
of $\Lambda$ and $M$.

\begin{lemma}\label{lemma:C-comodule_map}
Let $M$ be a complete $\acha$--comodule.
Then $\rho_{\Lambda,M}$ is a morphism
of $C$--comodules.
In other words, the following diagram commutes:
\[\xymatrix@C=50pt{
    M \ar[r]^{\rho_{\Lambda,M}} \ar[d]_{\rho_{C,M}} & \Lambda\wtens M \ar[d]^{\rho_{C,\Lambda\wtens M}}\\
    C\wtens M \ar[r]^-{1_C\otimes\rho_{\Lambda,M}} & C\wtens \Lambda\wtens M.
   }
\]    
Furthermore,
$\rho_{C,\Lambda\wtens  M}\circ \rho_{\Lambda,M}=
(1_C\otimes \rho_{\Lambda,M})\circ \rho_{C,M}$
is the $\acha$--comodule structure map $\rho_M$.
\end{lemma}

\begin{proof}
%
Let $f=(\pi_{\Lambda}\otimes \pi_C)\circ \psi\co \acha\to 
\Lambda\wtens  C$.
By the co-associativity of $\acha$--comodule $M$,
the following diagram commutes:
\[\xymatrix@C=50pt{
    M \ar[r]^{\rho_{\Lambda,M}} \ar[d]_{\rho_M} & \Lambda\wtens M \ar[d]^{1_{\Lambda}\otimes \rho_{C,M}}\\
    C\wtens \Lambda\wtens M \ar[r]^{f\otimes 1_M} & \Lambda\wtens C\wtens M.
   }
\]    
Let $\smash{g= (1_C\otimes 1_{\Lambda}\otimes 
\varepsilon_C)\circ\rho_{\smash{C,\Lambda\wtens C}}}\co
\Lambda\wtens C\to C\wtens \Lambda$.
Then we can check that $g\circ f$ is the identity map of
$C\wtens \Lambda$.
Since $(g\otimes 1_M)\circ (1_{\Lambda}\otimes \rho_{C,M})=
\smash{\rho_{C,\Lambda\wtens M}}$,
we obtain that $\rho_M=\rho_{C,\Lambda\wtens M}\circ 
\rho_{\Lambda,M}$.

Let $h=(\pi_C\otimes\pi_{\Lambda})\circ\psi\co
\acha\to \acha$.
By the coassociativity of $\acha$--comodule $M$,
the following diagram commutes:
\[\xymatrix@C=50pt{
    M \ar[r]^{\rho_{C,M}} \ar[d]_{\rho_M} & C\wtens M \ar[d]^{1_{C}\otimes \rho_{\Lambda,M}}\\
    C\wtens \Lambda\wtens M \ar[r]^{h\otimes 1_M} & C\wtens \Lambda\wtens M.
   }
\]    
But it is easy to check that $h$ is the identity
map of $\acha$.
Hence we obtain that
$\rho_M= (1_C\otimes\rho_{\Lambda,M})\circ \rho_{C,M}$.
This completes the proof.
\end{proof}

\begin{definition}
We say that a complete module $M$ is a
$C$--$\Lambda$--comodule if
$M$ is a  $C$--comodule and also a $\Lambda$--comodule
such that the structure map of $\Lambda$--comodule 
$\rho_{\Lambda,M}$ is a map of $C$--modules.  
\end{definition}

\begin{corollary}\label{cor:rchr2glambda}
A complete $\acha$--comodule has a natural $C$--$\Lambda$--comodule 
structure.
\end{corollary}

Let $\rho_{\Lambda,C\wtens \Lambda}=(\pi_{\Lambda}\otimes 1\otimes 1)
\circ\psi\co
C\wtens \Lambda\to \Lambda\wtens C\wtens \Lambda$.

\begin{lemma}\label{lemma:comutative_diagram_C_Lambda}
The following diagram commutes:
\[\xymatrix@C=50pt{
 \Lambda \ar[r]^{\psi_{\Lambda}} \ar[d]_{\rho_{C,\Lambda}} &  \Lambda\wtens \Lambda \ar[d]^{1_{\Lambda}\otimes\rho_{C,\Lambda}}\\
    C\wtens \Lambda \ar[r]^-{\rho_{\Lambda,C\wtens \Lambda}}
    &
    \Lambda\wtens C\wtens \Lambda.
 }
\]
\end{lemma}

\begin{proof}
Note that 
this is a 
diagram of $\fp{}$--algebras.
So it is sufficient to show the equality
$f(b_{(i)})=g(b_{(i)})$ holds for all $0\le i<n$,
where $f=(1_{\Lambda}\otimes\rho_{C,\Lambda})\circ \psi_{\Lambda}$
and $g=\rho_{\Lambda,C\otimes\Lambda}\circ\rho_{C,\Lambda}$.
We easily obtain that 
\[ f(b_{(i)})=b_{(i)}\otimes 1\otimes 1+
   1\otimes \sum_{j=0}^i \chi(s_{i-j})^{p^j}\otimes b_{(j)}.\]
On the other hand,
\[ \begin{array}{rcl}
    g(b_{(i)})&=&
    {\displaystyle \sum_{j=0}^i(1\otimes\chi(s_{i-j})^{p^j}\otimes 1)
     \cdot (1\otimes 1\otimes b_{(j)}
     +\sum_{k=0}^j b_{(k)}\otimes s_{j-k}^{p^k}\otimes 1)}\\
    &=&{\displaystyle \sum_{j=0}^i 1\otimes \chi(s_{i-j})^{p^j}
        \otimes b_{(j)}
        + \sum_{k=0}^i\sum_{j=k}^i b_{(k)}\otimes
          \left(s_{j-k}\chi(s_{i-j})^{p^{j-k}}\right)^{p^k}\otimes 1}\\
    &=&{\displaystyle \sum_{j=0}^i 1\otimes \chi(s_{i-j})^{p^j}
        \otimes b_{(j)}
        + b_{(i)}\otimes 1\otimes 1}.\\
   \end{array}\]
This completes the proof.
\end{proof}

\begin{corollary}\label{corollary:commutativity_left_bottom_corner}
If $M$ is a complete $C$--comodule,
then the following diagram commutes:
\begin{equation}\label{eq:commutativity_left_bottom_corner} 
\begin{aligned}
\xymatrix@C=60pt@R=40pt{
   \Lambda\wtens M \ar[r]^{\psi_{\Lambda}\otimes 1_M} \ar[d]_{\rho_{C,\Lambda\wtens M}} &   \Lambda\wtens \Lambda\wtens  M \ar[d]^{1_{\Lambda}\otimes \rho_{C,\Lambda\wtens M}}\\
    C\wtens \Lambda\wtens M \ar[r]^-{\rho_{\Lambda,C\wtens \Lambda}\otimes 1_M} &
    \Lambda\wtens C\wtens \Lambda\wtens M.
   }
\end{aligned}
\end{equation}
\end{corollary}

\begin{lemma}\label{lemma:big_9_diagram}
Let $M$ be a complete $C$--$\Lambda$--comodule
with $C$--comodule structure map $\rho_{C,M}\co M\to C\wtens M$. 
Then the following diagram commutes:
\[\xymatrix@C=65pt@R=38pt{
    M \ar[r]^-{\rho_{\Lambda, M}} \ar[d]^-{\rho_{\Lambda,M}} 
    & \Lambda\wtens M \ar[r]^-{\rho_{C,\Lambda\wtens M}} \ar[d]_-{1\otimes\rho_{\Lambda,M}} 
    & C\wtens \Lambda\wtens M \ar[d]^-{1\otimes 1 \otimes \rho_{\Lambda, M}}
    \\
    \Lambda\wtens M \ar[r]^-{\psi_{\Lambda}\otimes 1} \ar[d]^-{\rho_{C,\Lambda\otimes M}} 
    & \Lambda\wtens \Lambda\wtens M \ar[r]^-{\rho_{C,\Lambda\wtens \Lambda\wtens M}}  \ar[d]_-{1\otimes \rho_{C,\Lambda\wtens M}}
    & C\wtens \Lambda\wtens \Lambda\wtens M \ar[d]^-{1\otimes 1\otimes \rho_{C,\Lambda\wtens M}} 
    \\
    C\wtens \Lambda\wtens M \ar[r]^-{\rho_{\lambda,C\wtens \Lambda}\otimes 1} 
    & \Lambda\wtens C\wtens \Lambda \wtens M \ar[r]^-{\rho_{C,\Lambda\wtens C}\otimes 1\otimes 1}
    & C\wtens \Lambda\wtens C\wtens \Lambda\wtens  M
  }
\]
\end{lemma}

\begin{proof}
The top left square commutes since $M$ is a $\Lambda$--comodule.
{From} the assumption that $\rho_{\Lambda,M}$ is 
a morphism of $C$--comodules,
so is $1_{\Lambda}\otimes \rho_{\Lambda,M}$.
Hence we see that the top right square commutes.
The bottom left square commutes by 
\fullref{corollary:commutativity_left_bottom_corner}.
Since $\smash{\rho_{C,\Lambda\wtens M}}$ is 
a morphism of $C$--comodules,
so is $1_{\Lambda}\otimes \smash{\rho_{C,\Lambda\wtens M}}$.
Hence the bottom right square commutes.
This completes the proof.  
\end{proof}

\begin{lemma}\label{lemma:identification_comultiplication}
The map $(\rho_{C,\Lambda\wtens C}\otimes 1_{\Lambda})\circ
   (\rho_{\Lambda,C\wtens\Lambda})$
is the comultiplication $\psi$.
\end{lemma}

\begin{proof}
Let $f=(\smash{\rho_{\smash{C,\Lambda\wtens C}}}\otimes 1_{\Lambda})\circ
(\smash{\rho_{\smash{\Lambda,C\wtens\Lambda}}})$.
Since $f$ is a map of $\fp{}$--algebras,
it is sufficient to show that
$f(c)=\psi(c)$ for all $c\in C$
and $f(b_{(i)})=\psi(b_{(i)})$ for all
$0\le i<n$.
It is easy to check that $f(c)=\psi(c)$.
On the other hand,
\[ \begin{array}{rcl}
    f(b_{(i)})&=&{\displaystyle\sum_{j=0}^i\sum_{k=0}^j\sum_{l=0}^{i-j}
      \chi(s_{j-k})^{p^k}s_{i-j-l}^{p^j}\otimes b_{(k)}
      \otimes s_l^{p^{i-l}}\otimes 1} 
    + 1\otimes 1\otimes 1\otimes b_{(i)} \\
   &=&{\displaystyle\sum_{
      \stackrel{\mbox{$\scriptstyle k,l\ge 0$}}
               {\mbox{$\scriptstyle k+l\le i$}}}
    \left( \sum_{j=k}^{i-l}\chi(s_{j-k})s_{i-j-l}^{p^{j-k}}\right)^{p^k}
    \otimes b_{(k)}\otimes s_l^{p^{i-l}}}
   + 1\otimes 1\otimes 1\otimes b_{(i)} \\
   &=&1\otimes 1\otimes 1\otimes b_{(i)}
   + {\displaystyle\sum_{k=0}^i 1\otimes b_{(k)}\otimes 
          s_{i-k}^{p^k}\otimes 1}.\\ 
   \end{array}\]
Hence $f(b_{(i)})=\psi(b_{(i)})$.
This completes the proof.
\end{proof}

Let $M$ be a complete $C$--$\Lambda$--comodule with
$C$--comodule structure map $\rho_{C,M}\co M\to C\wtens M$. 
We define a map $\rho_M\co M\to C\wtens \Lambda\wtens M$ by
\[ \rho_M\co M\stackrel{\rho_{\Lambda,M}}{\hbox to 10mm{\rightarrowfill}}
   \Lambda\wtens M\stackrel{\rho_{C,\Lambda\wtens M}}
   {\hbox to 10mm{\rightarrowfill}} C\wtens \Lambda\wtens M .\]
By \fullref{lemma:big_9_diagram} and 
\fullref{lemma:identification_comultiplication},
we see that $\rho_M$ gives $M$ a complete $\acha$--comodule structure.

\begin{proposition}
Let $M$ be a complete $C$--$\Lambda$--comodule.
Then $M$ has a natural $\acha$--comodule structure $\rho_M$ such that
the induced $C$--$\Lambda$--comodule structure coincides with
the given one.
\end{proposition}

Note that if $M$ is a complete $C$--$\Lambda$--comodule
obtained from a complete $\acha$--comodule,
then the induced $\acha$--comodule structure coincides 
with the given one by \fullref{lemma:C-comodule_map}.  

By summarizing the results in this section,
we obtain the following theorem.

\begin{theorem}\label{thm:eq_cat_rchr_stab_lambda}
There is an equivalence of symmetric monoidal categories
between the category of complete $\acha$--comodules
and the category of complete $C$--$\Lambda$--comodules.
\end{theorem}

\section{Main theorem}
\label{section:Main_Theorem}

\subsection[Symmetric monoidal functor F]{Symmetric monoidal functor ${\mathcal F}$}
\label{subsection:symmetric_monoidal_dunctor}

In \fullref{prop:proftoprof} 
we showed that ${\mathcal F}$ is a monoidal functor 
from the category of profinite $C_{E_*}$--precomodules
to the category of profinite $C_{K_*}$--comodules.
In this section we show that the functor 
${\mathcal F}$ extends to a monoidal functor 
from the category of profinite $\ee$--precomodules
to the category of profinite $\kk$--comodules.

We let $M$ be a profinite $\ee$--precomodule.
Then $M$ is a profinite $C_{E_*}$--preco\-mod\-ule and
also a $\Lambda_{E_*}$--comodule such that
the $\Lambda_{E_*}$--comodule structure map
$\rho_M\co M\to \Lambda_E\wtens M$ is a map
of profinite $C_{E_*}$--precomodules by \fullref{cor:rchr2glambda}.
We note that 
$\Lambda_{E_*}$ is a $C_{E_*}$--precomodule 
and 
there is an isomorphism of $C_{K_*}$--comodules:
${\mathcal F}(\Lambda_{E_*})\cong \Lambda_{K_*}$
by \fullref{thm:iso_bet_lambdas}.

\begin{lemma}\label{lemma:def_comodule_str_map}
If $M$ is a profinite $C_{E_*}$--precomodule,
then the natural map
\[ \Lambda_{K_*}\wtens_{K_*}{\mathcal F}(M)
   \stackrel{\cong}{\longrightarrow} 
   {\mathcal F}(\Lambda_{E_*})\wtens_{K_*}{\mathcal F}(M)
   \longrightarrow 
   {\mathcal F}(\Lambda_{E_*}\wtens_{E_*}M)\]
is an isomorphism of $C_{K_*}$--comodules.
\end{lemma}

\begin{proof}
Since
${\mathcal F}(\Lambda_{E_*})$ is isomorphic to $\Lambda_{K_*}$
as a $C_{K_*}$--comodule,
$\Lambda_{E_*}\wtens_{E_*}L_*\cong\Lambda_{K_*}\wtens_{K_*}L_*$ 
as twisted $L_*$--${\mathcal G}$--modules.
Hence there are isomorphisms of twisted $L_*$--${\mathcal G}$--modules:
\[ \begin{array}{rcl}
     \Lambda_{E_*}\wtens_{E_*} M\wtens_{E_*} L_* & \cong &
     \left(\Lambda_{E_*}\wtens_{E_*}L_*\right)
     \wtens_{L_*} \left(M\wtens_{E_*}L_*\right)\\[1mm]
     &\cong& \left(\Lambda_{K_*}\wtens_{K_*}L_*\right)
     \wtens_{L_*}\left(M\wtens_{E_*}L_*\right)\\[1mm]
     &\cong& \Lambda_{K_*}\wtens_{K_*}M\wtens_{E_*}L_*.\\ 
   \end{array}\]   
By taking $S_{n+1}$--invariant submodules,
we obtain an isomorphism of twisted $K_*$--$G_n$--modules:
${\mathcal F}(\Lambda_{E_*}\wtens_{E_*}M)\cong
   \Lambda_{K_*}\wtens_{K_*}{\mathcal F}(M)$.
This implies that the above map is an isomorphism of 
profinite $C_{K_*}$--comodules.
\end{proof}

By \fullref{lemma:def_comodule_str_map},
we obtain a map 
\[ {\mathcal F}(\psi_{\Lambda_{E_*}})\co
   \Lambda_{K_*}\cong {\mathcal F}(\Lambda_{E_*})\longrightarrow 
   {\mathcal F}(\Lambda_{E_*}\wtens_{E_*}\Lambda_{E_*})
   \cong \Lambda_{K_*}\wtens_{K_*}\Lambda_{K_*}.\] 

\begin{lemma}
The map ${\mathcal F}(\psi_{\Lambda_{E_*}})$ coincides
with the comultiplication map $\psi_{\Lambda_{K_*}}$
on $\Lambda_{K_*}$.
\end{lemma}

\begin{proof}
This follows from the fact that
the algebra generators $\smash{\widehat{b}^E_{(i)}}$
of $\smash{{\mathcal F}(\Lambda_{E_*})}$ 
are given by linear combinations of 
the algebra generators $\smash{b^E_{(i)}}$ of $\Lambda_{E_*}$
with coefficients in $L_*$
(see \fullref{lemma:iso_f_lambda2lambda}). 
\end{proof}

\begin{corollary}
Let $M$ be a profinite $\ee$--precomodule with corresponding $\Lambda_{E_*}$--comodule structure map
$\rho_M\co M\to \Lambda_{E_*}\wtens_{E_*} M$.
Then the map ${\mathcal F}(\rho_M)\co {\mathcal F}(M)\to 
{\mathcal F}(\Lambda_{E_*}\wtens_{E_*} M)\cong \Lambda_{K_*}\wtens_{K_*} 
{\mathcal F}(M)$ defines a natural $\Lambda_{K_*}$--comodule structure 
on ${\mathcal F}(M)$.
\end{corollary}

\begin{proposition}\label{prop:natural_str_kk_comod}
If $M$ is a profinite $\ee$--precomodule,
then ${\mathcal F}(M)$ has a natural $\kk$--comodule structure.
\end{proposition}

\begin{proof}
Since ${\mathcal F}$ is a functor
and the $\Lambda_{E_*}$--comodule structure map
$\rho_M\co M\to\Lambda_E\wtens M$ is 
a map of $C_{E_*}$--precomodule,
${\mathcal F}(\rho_M)$ is a map of $C_{K_*}$--comodules.
Hence the proposition follows from 
\fullref{thm:eq_cat_rchr_stab_lambda}.
\end{proof}

\begin{corollary}\label{cor:comodule2comodule}
${\mathcal F}$ extends
to a symmetric monoidal functor from
the category of profinite $\ee$--precomodules
to the category of profinite $\kk$--comodules.  
\end{corollary}

\begin{proof}
By \fullref{prop:natural_str_kk_comod},
we see that ${\mathcal F}$ extends to a functor
from the category of profinite $\ee$--precomodules
to the category of profinite $\kk$--comodules.   
It is easy to check that ${\mathcal F}$ 
respects the symmetric monoidal structures.
\end{proof}

\subsection{Main theorem}
\label{subsection:Main_Theorem}

In this section we prove the main theorem 
(\fullref{thm:Main_Theorem}).
The theorem states that for any spectrum $X$,
${\mathcal F}(\mathbb{E}^*(X))$ is naturally isomorphic to 
$\mathbb{K}^*(X)$ as a cofiltered system of finitely generated
discrete $\kk$--comodules.
Furthermore, if $X$ is a space,
then this equivalence respects the graded commutative ring
structures.

\begin{definition}\rm
Let $\smash{\mathcal{M}_{\mathbb{E}}^f}$ 
(resp.\ $\smash{\mathcal{M}_{\mathbb{K}}^f}$) 
be the category of finitely generated discrete 
$\ee$--precomodules (resp.\ $\kk$-(pre)comodules).
We define $\mathcal{M}_{\mathbb{E}}$
(resp.\ $\mathcal{M}_{\mathbb{K}}$)
to be the procategory of $\smash{\mathcal{M}_{\mathbb{E}}^f}$
(resp.\ $\smash{\mathcal{M}_{\mathbb{K}}^f}$).
\end{definition}

By \fullref{cor:comodule2comodule},
we can extend the functor $\mathcal{F}$
from $\mathcal{M}_{\mathbb{E}}$ to $\mathcal{M}_{\mathbb{K}}$
by obvious way:
\[ \mathcal{F}\co \mathcal{M}_{\mathbb{E}}\longrightarrow
                \mathcal{M}_{\mathbb{K}}.\]
Note that $\mathcal{F}$ is a monoidal functor.
As in the cases of $\mathcal{C}_{\mathbb{E}}$
and $\mathcal{C}_{\mathbb{K}}$,
we have the following lemma.

\begin{lemma}\label{lemma:precomodule_str_e_coh}
For any spectrum $X$,
$\mathbb{E}^*(X)\in\mathcal{M}_{\mathbb{E}}$ and
$\mathbb{K}^*(X)\in \mathcal{M}_{\mathbb{K}}$. 
\end{lemma}

Hence 
${\mathcal F}(\mathbb{E}^*(X))\in \mathcal{M}_{\mathbb{K}}$.
The natural $\Lambda_{K_*}$--comodule structure on 
${\mathcal F}(\mathbb{E}^*(X))$ gives natural $K_*$--module homomorphisms
$\widehat{Q}_i$ on $K^*(X)=\smash{\inverselimit{}}
\mathcal{F}(\mathbb{E}^*(X))$ for $0\le i<n$ 
with respect to the algebra generators $b_{(i)}^K$ of $\Lambda_{K_*}$.

\begin{lemma}\label{lemma:wideq_stable_op}
For $0\le i<n$, 
$\widehat{Q}_i$ is a stable cohomology operation on $K^*(X)$.
\end{lemma}

\begin{proof}
It is sufficient to show that $\widehat{Q}_i$ commutes with
the suspension isomorphism $\Sigma$.
Let $s\in \smash{\widetilde{K}^1(S^1)}$ the canonical generator $\Sigma (1)$.
Then the suspension isomorphism is given by
the (exterior) product with $s$.
Since $\smash{\widehat{Q}_i}$ is an odd degree operation,
$\smash{\widehat{Q}_i}$ acts on $s$ trivially.
Hence we see that $\smash{\widehat{Q}_i}$ commutes with the product with $s$
since $\smash{\widehat{Q}_i}$ is a derivation.
This completes the proof.
\end{proof}

Recall that ${\mathcal Y}$ is the lens space $S^{2p^n-1}/C_p$,
and $K^*({\mathcal Y})=\Lambda(u_k)\otimes K_*[x_K]/(x_K^{p^n}).$

\begin{lemma}\label{lemma:effect_qi}
For $0\le i<n$,
$\widehat{Q}_i(u_K)=x_K{}^{p^i}$.
\end{lemma}

\begin{proof}
By \fullref{lemma:action_on_yA},
$\rho(y_E)=1\otimes y_E + b^E(x_E)$.  
Since $b_E(X)=\widehat{b}(\Phi(X))$ mod $(X^{p^n})$ by definition,
we have
$\rho(1\otimes u_E) =1\otimes1 \otimes u_E + \smash{\widehat{b}}^E(\Phi(x_E))$.  
{From} that fact that 
$u_K=1\otimes u_E$ and $\Phi(x_E)=x_K$, 
we obtain that 
$\smash{\widehat{Q}_i(u_K)= x_K{}^{p^i}}$.
\end{proof}

\begin{corollary}\label{cor:equality_q}
For $0\le i<n$,
$\widehat{Q}_i=Q_i^K$.
\end{corollary}

\begin{proof}
By \fullref{lemma:wideq_stable_op},
$\widehat{Q}_i$ is an odd degree stable cohomology operation,
which is a derivation with respect to the (exterior) product.
Hence $\smash{\widehat{Q}_i}$ is characterized by the action
on $u_K\in K^1({\mathcal Y})$.
Then the corollary follows from \fullref{lemma:effect_qi}.   
\end{proof}

Recall that 
the generalized Chern character 
\eqref{eq:generalized_Chern_character}
\[ \Theta\co E^*(X)\longrightarrow L^*(X) \]
induces a natural isomorphism in $\mathcal{C}_{\mathbb{K}}$
\[ {\mathcal F}(\mathbb{E}^*(X))\cong \mathbb{K}^*(X) \]
by \fullref{thm:reformulation_stab}.
The following is our main theorem of this note.

\begin{theorem}\label{thm:Main_Theorem}
The generalized Chern character $\Theta$ 
induces a natural isomorphism in $\mathcal{M}_{\mathbb{K}}$:
\[ {\mathcal F}(\mathbb{E}^*(X))\cong \mathbb{K}^*(X). \]
If $X$ is a space, then this is an isomorphism 
of cofiltered systems of finite $\kk$--comodule algebras. 
\end{theorem}

\begin{proof}
By \fullref{thm:reformulation_stab},
there is a natural isomorphism 
${\mathcal F}(\mathbb{E}^*(X))\cong \mathbb{K}^*(X)$ 
in $\mathcal{C}_{\mathbb{K}}$.
\fullref{cor:equality_q}
implies that the isomorphism 
$\mathcal{F}(E^*(Z))\cong K^*(Z)$
respects the $\Lambda_{K_*}$--comodule structures
for all $Z\in\Lambda(X)$.
Hence 
the theorem follows from 
\mbox{\fullref{thm:eq_cat_rchr_stab_lambda}}.
\end{proof}

\bibliographystyle{gtart}
\bibliography{link}

\end{document}